\documentclass{svjour3}
\usepackage{color} 
\usepackage{amsmath,amstext,amssymb}

\usepackage{tikz}
\usepackage[colorinlistoftodos,prependcaption,textsize=small]{todonotes}

\newcommand{\exclude}[1]{}
\def \conv{\textup{conv}}

\def \FF{ {\mathcal{F}}}

\def \SS{ {\mathcal{S}}}
\def \TT{ {\mathcal{T}}}

\def \B{{\mathcal{B}}}
\def \C{{\mathcal{C}}}

\def \rr{ {\mathbb{R}}}
\def \R{\rr}
\def \zz{ {\mathbb{Z}}}
\def \Z{\zz}

\newcommand{\tred}[1]{\texttt{\textcolor{red} {[#1]}}}

\renewcommand{\qed}{~~~~\hfill\rule{.05in}{.05in}}

\newcommand{\ceil}[1]{\left\lceil #1\right\rceil}

\newcommand{\seq}{\subseteq}

\newcommand {\beq}{\[}\newcommand {\eeq}{\]}
\newcommand {\beqn}{\begin{equation}}\newcommand {\eeqn}{\end{equation}}
\newcommand {\beqan}{\begin{eqnarray}}\newcommand {\eeqan}{\end{eqnarray}}
\newcommand {\beqa}{\begin{eqnarray*}}\newcommand {\eeqa}{\end{eqnarray*}}

\newcommand \ctx[2]{#1^T#2}
\newcommand \projk[2]{\textup{proj}_{#1}(#2)}
\def \tq{ {\,:\,}}
\def \SC{\textup{SC}}

\DeclareMathOperator    \proj      {proj}
\newcommand {\ignore}[1]{}

\fontsize{53}{200}
\renewcommand{\ell}{l}

\newcommand{\BF}{B^{\text{F}}\hspace{-0.04cm}}
\newcommand{\BU}{B^{\text{U}}\hspace{-0.04cm}}
\newcommand{\BLG}{B^{\text{L}}\hspace{-0.04cm}}
\newcommand{\BLGp}{B^{\text{L}^+}\hspace{-0.04cm}}

\newcommand{\BALT}{B^{\text{o}}\hspace{-0.04cm}}
\def \PLG{P_{L}}
\def \PLGp{P_{L+}}

\begin{document}  \large	 

\title{Binary Extended Formulations}
\author{Sanjeeb Dash \and Oktay G\"unl\"uk \and Robert Hildebrand}
\authorrunning{  Dash,    G{\"u}nl{\"u}k,   Hildebrand}
\institute{IBM T. J. Watson Research Center, Yorktown Heights, New York}
\date{\today}
\maketitle

\begin{abstract}
  We analyze different ways of constructing binary extended formulations of mixed-integer problems with bounded integer variables and compare their relative strength with respect to split cuts. 
  We show that among all binary extended formulations where each bounded integer variable is represented by a distinct collection of binary variables, what we call ``unimodular" extended formulations are the strongest.
  We also compare the strength of some binary extended formulations from the literature.
  Finally, we study the behavior of branch-and-bound on such extended formulations and show that branching on the new binary variables leads to significantly smaller enumeration trees in some cases.
\end{abstract}


\section{Introduction}\label{sec-intro}
\ignore{An extended formulation of {a mathematical programming}  formulation of a  optimization problem is one which uses additional variables to represent the same problem.
In integer programming, it is common to use extended formulations such that the linear programming (LP) relaxation of the extended
formulation is ``stronger'' (i.e., contained in) than the LP relaxation of original integer program (IP). 
(At one extreme, the extended formulation may be an LP such that the projection of its feasible region equals the integer hull of the original problem, see \cite{RWZ} for references to recent work on this topic.) 
}
For a given formulation of an optimization problem, an extended formulation  is one which uses additional variables to represent the same problem.
In integer programming, it is common to use extended formulations that lead to stronger  LP relaxations.
(Ideally, the extended formulation may have an LP relaxation whose projection onto the original space is integral, see \cite{RWZ} for references to recent work on this topic.) 
For binary integer programs, the lift-and-project methods of 
Sherali and Adams \cite{SA90}, Lov\'{a}sz and Schrijver \cite{LS91},  and Balas, Ceria and Cornu\'ejols \cite{BC93} yield such extended formulations. 
However, these extended formulations are, in general, too big to be practically useful as are those given by Bodur, Dash and G\"unl\"uk \cite{bdg} for general integer programs.

In this paper, we\ignore{analyze a number of previously studied} study extended formulations of bounded integer programs that are constructed by representing \ignore{each bounded} integer variables by a\ignore{linear} combination of new binary variables, possibly along with additional constraints on these binary variables.
Such ``binary extended formulations'' have been studied by Glover \cite{glover}, Sherali and Adams \cite{SA99}, and Roy \cite{roy}.
Given a polyhedral mixed-integer set 
\begin{equation}\label{eqp} P = \{ (x,y) \in U \times \rr^n: Ax + Cy \leq b\}
\end{equation}
where $A,C,b$ are matrices of appropriate dimension and $U=\{0, \ldots, u_1\}\times\cdots\times\{0, \ldots, u_l\}$ with $u_1,\ldots,u_l\in \Z$,  
Sherali and Adams \cite{SA99} studied the binary extended formulation:
\beqan
  Q = \{ (x,y,z) \in \rr^{l} \times \rr^n \times \{0,1\}^{q} : Ax + Cy &\leq& b,\nonumber\\ 
  x_i = \sum_{j=1}^{u_i}jz_{ij}, \sum_{j=1}^{u_i} z_{ij} \leq 1 &\mbox{for}& i=1, \ldots, l\}\label{eqq}
    \eeqan
where $q=\sum_{i=1}^lu_i$. 
For $i\in\{1,\ldots,l\}$, the binary variables $z_{ij}$ for $j\in\{0,\ldots,u_i\}$ are used to ``binarize" variable $x_i$.
Note that there is a one-to-one mapping between each $x_i \in \{0,\ldots, {u_i}\}$ and each  $(z_{i1}, \ldots, z_{i{u_i}}) \in \{0,1\}^{u_i}$ satisfying  $\sum_{j=1}^{u_i} z_{ij} \leq 1$ and $x_i = \sum_{j=1}^{u_i} jz_{ij}$.
%
%
More generally, Roy \cite{roy} defined a binary extended formulation of $P$ to be a set $S$ of the form
\begin{equation}\label{eqs}
  S = \{ (x,y,z) \in \rr^{l} \times \rr^n \times \{0,1\}^{q} : Ax + Cy \leq b, x = Tz, Dz \leq f \},
 \end{equation}
for some $q > 0$, and some matrices $D,T, f$, where the linear mapping $x = Tz$ maps 0-1 points in $\{z \in \rr^q: Dz \leq f\}$ to  $U$.
In this paper we will study reformulations where each bounded integer variable is ``binarized'' separately, i.e., it is represented by a distinct collection of binary variables.
Note that in both cases above, only the obvious domain of $x_i$ is used to binarize $x_i$, and the constraints $Ax + Cy \leq b$ do not play a role.

Owen and Mehrotra \cite{OM} proved some negative properties of two such binary extended formulations vis-a-vis the original integer program.   
In particular, they showed that 0-1 branching would perform worse on the extended formulation than on the original integer program in the sense that a much larger branch-and-bound tree would be generated in the former case, unless one branches in a specific manner. They thus argue that such binarization strategies are unlikely to be useful. 

However, binary extended formulations have some attractive theoretical properties with respect to cutting planes.  
Cook, Kannan, and Schrijver \cite{CKS} gave a mixed-integer program (MIP) with two bounded (between 0 and 2), integer variables  and one bounded, continuous variable that 
cannot be solved in finite time by any cutting plane algorithm that only generates split cuts.
But the binary extended formulation (\ref{eqq}) can be solved in finite time with split cuts as all binary programs have this property, see  Balas \cite{B79}.

Bonami and Margot \cite{BM} showed that certain types of cutting planes were more effective (both theoretically and computationally) when generated on a binary extended formulation as opposed to the original formulation. More strikingly, Angulo and Van Vyve \cite{AvV} showed that CPLEX \cite{cpx} requires significantly more time to solve an MIP formulation of the flow cover problem than a particular binary extended formulation (unlike Roy \cite{roy} and Sherali and Adams \cite{SA99}, they use the constraints $Ax + Cy \leq b$ of the mixed-integer set to construct the extended formulation).

In practice, binarization changes the behavior of  MIP solvers both in terms of  branching and  cut generation.
In this paper we consider known binary extended formulations as well as more general ways to construct them and compare their relative strength  with respect to adding certain families of cutting planes in the extended space.
For some pairs of previously studied binary extended formulations, we show that the projection of the split closure of one extended formulation onto the original space of variables is strictly contained in the corresponding projection of the other.
A natural question is whether it is possible to construct a strongest possible -- in the above sense -- binary extended formulation.
{ Our main result is that among all binary extended formulations where each bounded integer variable is separately binarized, what we call ``unimodular" extended formulations are strongest with respect to the projection of their split closures.
Both the formulation in (\ref{eqq}) and the extended formulation studied by Roy \cite{roy} and Bonami-Margot \cite{BM} belong to this class.
Finally, we study the behavior of branch-and-bound on a certain binary extended formulation and show that the observation by Owen and Mehrotra \cite{OM} does not always hold.}

The rest of the paper is organized as follows: in the next section, we formally define binary extended formulations and review split cuts.
In Section \ref{sec-basic}, we study basic properties of binary extended formulations. 
In Section \ref{sec:compare}, we compare a number of binary extended formulations in terms of the strength of their split closures.  
Finally, in Section \ref{sec:branch} we show that branching in the extended space can lead to smaller branch-and-bound trees when solving a mixed-integer program.

\section{Preliminaries}\label{sec-prelim}
{ We next formally define what we mean by binarization polytopes, binarization schemes, and binary extended formulations.
We also review split cuts and define unimodular and integral affine transformations.}

\subsection{Notation}\label{sec-not}
Let $P \subseteq \rr^n$ be a rational polyhedron  (all polyhedra in this paper are assumed to be rational).
 { Let  $I = \{1, \ldots, l\}$ be the index set of integer variables where $0\le l \leq n$.}
We call a set of the form \[P^I =  \{x \in P :  x_i \in \zz, \mbox{ for } i \in I\} \]
a \emph{polyhedral mixed-integer set}, and we call $P$ the linear relaxation of $P^I$.
For convenience, we assume that all variables defining $P^I$ are bounded, i.e.,
$P$ is defined by rational data as
\begin{equation} P=\Big\{ x\in\R^n\::\: Ax\le b,~ 0\le x_i\le u_i\text { for }i\in I\Big\}. \label{eq:p}\end{equation}
A polyhedral set $X \subseteq \R^n \times \R^q$ is called an \emph{extended formulation} of $P$ if 
$P= \proj_{x}(X)$ where $\projk{x}{X}$ stands for the orthogonal projection of points in $X$ to the space of the variables $x$.

For positive integers $q,u$, let $\Gamma^q_u$ be the set of all rational polytopes $B\subseteq\{(x,z)\in\R\times[0,1]^q\::\: 0\le x\le u\}$ such that 
\begin{equation}\label{eq:gamma} \proj_{x} \{B \cap (\R \times \{0,1\}^q)\} = \{0,1,\ldots,u\}.\end{equation}
Each polytope in $\Gamma^q_u$ can be used to ``binarize" a bounded integer variable $x \in \{0, \ldots, u\}$ using $q$ new binary variables; setting the new variables to 0-1 values forces $x$ to be an integer in $\{0, \ldots, u\}$.
We  refer to each polytope in $\Gamma^q_u$ as a {\em binarization polytope}.
We note that due to \eqref{eq:gamma}, if $B\in \Gamma^q_u$, then $B \cap (\R \times \{0,1\}^q)$ might contain points of the form $(x,z)$ and $(x,z')$ where $x\in \{0, \ldots, u\}$ and $z\not=z'$.
However, $B \cap (\R \times \{0,1\}^q)$ does  not contain two points of the form $(x,z)$ and $(x',z)$ where $x\not=x'$ as this would imply that the segment $\conv(\{x,x'\})$ belongs to the projection which contradicts the fact that \eqref{eq:gamma} is a discrete set.

The following are some  examples of binarization polytopes:
\beqan
\BF(u) &=& \{(x,z)\in\R\times[0,1]^{u}\::\:\textstyle x=\sum_{j=1}^{u} jz_j,~\sum_{j=1}^{u} z_j\le1\}, \label{bin-sa}\\[.2cm]
\BU(u)& =& \{(x,z)\in\R\times[0,1]^{u}\::\:\textstyle x=\sum_{j=1}^{u}  z_j,~1\ge z_1\ge  z_2\ge\ldots \ge z_u \ge 0\}, \label{bin-roy}\\[.2cm]
\BLG(u)& =& \{(x,z)\in\R\times[0,1]^{\ceil{ \log_2(u+1)}}\::\:\textstyle x = \sum_{j=0}^{\ceil{ \log_2(u+1)}-1} 2^j z_j\}. \label{bin-log}
\eeqan 
Note that sets  $\BF(u)$ and  $\BU(u)$ are contained in $\Gamma^u_u$, whereas $\BLG(u)$ has
only $\ceil{ \log_2(u+1)}+1$ variables. 
The set $ \BF(u)$, known as the {\em full}-binarization,  was studied by Sherali and Adams \cite{SA99} and by Angulo and Van Vyve \cite{AvV}.
The {\em unary}-binarization  $ \BU(u)$ was studied by Roy \cite{roy} and by Bonami and Margot \cite{BM}.
The {\em logarithmic}-binarization $\BLG(u)$ was studied by Owen and Mehrotra \cite{OM}.

Note that our definition does not require a bijection between integer points in $B\in\Gamma^q_u$ and $\{0,1,\ldots,u\}$.
However, we will later show that this is a desirable property and  is satisfied by  $\BF(u)$,  $\BU(u)$, and  $\BLG(u)$. Also note that $B \cap (\R \times \{0,1\}^q)$ contains at least $u$ distinct points and therefore $q\ge \ceil{ \log_2(u+1)}$.


Let $\mathcal{B} = (B^1, \ldots, B^l)$ be an ordered set of $l$ polytopes where each $B^i \in \Gamma^{q_i}_{u_i}$.
We will call ordered sets of the form $\mathcal{B}$ {\em binarization schemes} and in particular if all $B^i$ defining $\mathcal{B}$ are unary (or full or logarithmic) binarization polytopes, we will call the scheme a unary (respectively, full or logarithmic) binarization scheme.
Let $q=\sum_{i\in I} q_i$.
We define $P_{\B}$ to be the polyhedron
\begin{equation}\label{eq:pb}
P_{\B}=\Big\{ (x,z)\in\R^n\times\R^{q}\::\: x\in P,~ (x_i,z_i)\in B^{i} \text { for }i\in I\Big\}.
\end{equation}
Here we abuse notation, and let $z$ be a vector in $\R^q$, and $z_1 \in \R^{q_1}, \ldots, z_l\in\R^{q_l}$ be subvectors of $z$; i.e., $z^T = (z_1^T, \ldots, z_l^T)$, and $z_{ij}$ is the $j$th component of the $i$th subvector of $z$. 
$P_\B$ is an extended formulation of $P$, i.e., $\proj_x(P_\B) = P$ since for every $x \in P$ and $i \in I$, $0 \leq x_i \leq u_i$, and hence there exists $z_i\in \R^{q_i}$ such that $(x_i, z_i) \in B^i$. 
Let \beqn I_\B = \{1, \ldots, l, n+1, \ldots, n+q\}.\label{def:IB}\eeqn
We call $P_\B^{I_\B}$ a {\em binary extended formulation} 
of $P$; the integrality requirements on the new $z$ variables force the $x$ variables to be integral.
Therefore one can drop the integrality requirements on the $x$ variables in $P_\B^{I_\B}$, and get a valid extended formulation of $P^I$.
However, we will later argue that one may be able to obtain stronger split cuts by retaining the integrality of the $x$ variables.

\subsection{Integral, affine transformations and split cuts}\label{sec-uni}
For a given set $X \subseteq \rr^n$, we denote its convex hull by $\conv(X)$.
Let $P \subseteq \rr^n$ be a rational polyhedron, and let $1 \le l \leq n$ and $I = \{1, \ldots, l\}$.
Given $(\pi,\pi_0)\in\zz^n\times\zz$, the {\em split set} associated with $(\pi,\pi_0)$ is defined to be
$$ S(\pi,\pi_0)=\{x\in\rr^n\tq \pi_0<\ctx{\pi}{x}<\pi_0+1\}.$$
We call a valid inequality for $\conv(P\setminus S(\pi,\pi_0))$  a {\em split cut} for $P$ {\em derived from} $S(\pi,\pi_0)$.
If $\pi\in\zz^{l}\times \{0\}^{n-l}$ and $\pi_0 \in \zz$, then
$\zz^{l}\times\rr^{n-l}\subseteq \rr^n\setminus S(\pi,\pi_0), $
and split cuts derived from the associated split set are valid for $P^I$.
Let $\SS_n(I) = \{S(\pi, \pi_0) : \pi\in\zz^{l}\times \{0\}^{n-l}, \pi_0 \in \zz\}$.
We define the {\em split closure of $P$ with respect to $I$} as
\[ \SC(P,I)=\bigcap_{S\in\SS_n(I)}\conv\left(P\setminus S\right). \]
It is easy to see that for all $P,Q \subseteq \R^n$,
\beqn P \subseteq Q \implies  \SC(P,I) \subseteq \SC(Q,I).\label{containment}\eeqn

For $k=2, 3,\ldots$, we define $\SC^k(P,I) = \SC(\SC^{k-1}(P,I), I)$ where $\SC^1(P,I) = \SC(P,I)$.  
Split closures were first studied in ~\cite{CKS} and play an important role in the theory and practice of integer programming.


For a given polyhedral mixed-integer set $P^I$ and two binarization schemes $\mathcal{B}$ and $\mathcal{C}$, we want to compare the ``strength'' of the associated extended formulations $P_{\B}$ and $P_{\C}$ after applying the split closure operation.
As $P_{\B}$ and $P_{\C}$ may not belong to the same Euclidean space, we compare the projections of their split closures onto the original space. 

A  function $f: \rr^n \rightarrow \rr^n$ is a {\em unimodular transformation} if  $f(x) = Ux + v$ where $U$ is a $n \times n$ unimodular matrix (i.e., an integral matrix with determinant $\pm 1$) and $v \in \zz^n$.
The split closure operation is invariant under unimodular transformations, see \cite[Proposition 3]{dgl} and also \cite{dey}.
We generalize this result in Theorem~\ref{thm:trans} by giving a 
result on integral, affine transformations, i.e., functions $f : \rr^m \rightarrow \rr^n$ of the form $f(x) = Vx + v$ where $V$ is an integral $n \times m$ matrix, and $v \in \zz^n$.
For such an $f$,
and $S \subseteq \rr^n$, we define $f^{-1}(S) = \{x \in \rr^m : f(x) \in S\}$,
and for a collection  $\mathcal S$ of subsets of $\R^n$,  we define $f^{-1}(\mathcal S) = \{ f^{-1}(S) : S \in \mathcal S\}$.

\section{Basic properties of binarizations}\label{sec-basic}
In this section, we study how to get a stronger relaxation than the split closure of a polyhedron
by applying split cuts to a binary extended formulation.

{
Let $P$ be defined as in (\ref{eq:p}) and let $P_\B$ be defined as in (\ref{eq:pb}).
We observed in Section~\ref{sec-not} that $\projk{x}{P_\B} = P$.
It is shown in \cite{bdg} that if the new variables in $P_\B$ are treated as continuous variables, then the projection of the split closure of $P_\B$ is contained in the split closure of $P$. 
In other words, $\projk{x}{\SC(P_\B, I)} \subseteq \SC(P, I)$. 
Even though the containment can be strict for  extended formulations in general (without declaring the new variables integral),  we next show that this is not the case for $P_\B$.
}
For $i\in I$ let $w^i_0,w^i_{u_i}\in \{0,1\}^{q_i}$ be such that  
$(0,w^i_0),(u_i,w^i_{u_i})\in B_i$.
Clearly, $T^i=\conv\{(0,w^i_0),(u_i,w^i_{u_i})\}$ is contained in $B^i$ and is 1 dimensional.
Furthermore,
\beq
P_{\mathcal T}=\Big\{ (x,z)\in\R^n\times\R^{q}\::\: x\in P,~ (x_i,z_i)\in T^{i} \text { for }i\in I\Big\}
\eeq
is an extended formulation of $P$ contained in $P_{\B}$.
Since $P_{\mathcal T}$  has the same dimension as $P$, by \cite[Corollary 4.5]{bdg} we have $\projk{x}{\SC(P_\B, I)}\supseteq\projk{x}{\SC(P_\mathcal T, I)} = \SC(P, I)$.
Therefore, to get stronger split cuts from $P_\B$, the new variables should be explicitly declared as binary variables.

Another natural question is whether the original variables need to be declared integral in the extended formulation as the integrality of the new $z$ variables implies the integrality of the original $x$ variables. 
We will next argue that in some cases, this is necessary as  one gets weaker split cuts otherwise.

\subsection{Linear binarizations}
Let $B \in \Gamma^q_u$ be a binarization polytope. 
We say that $B$ is {\em affine} if all $(x,z)\in B$  satisfy  $x = \alpha^Tz + \alpha_0$ for some $\alpha \in \R^n$ and $\alpha_0 \in \R$; in this case
for all $(j,w^j)\in B$ with $w^j \in \{0,1\}^q$, we have $j = \alpha^T w^j +\alpha_0$. 
We call $B$ {\em linear} if it is affine and $\alpha_0 = 0$.
%
%
The binarization polytopes $\BU(u), \BF(u)$ and $\BLG(u)$ are all linear.
However not all binarizations are affine, as we show in the next example.
\begin{example}
Consider the binarization polytope associated with $x \in \{0,\ldots,3\}$ given by
\beqan\label{eq:non-linear}
B &=& \conv \left\{
(0,~ (0, 0)), (1,~ (1,0)), (2,~ (1, 1)), (3,~ (0,1)) \right\}\\
&=& \big\{(x ,z)\in [0,3] \times [0,1]^2: x - z_1 - z_2 \geq 0,~~~~ 	x + z_1 - 3 z_2 \geq 0,\nonumber\\
&&\hskip3.5cm	-x + z_1 + 3 z_2 \geq 0,~~	-x - z_1 + z_2 \geq -2 \big\}.\nonumber 
\eeqan
$B$ is not affine.  If it were, then for some $\alpha \in \R^2$ and $\alpha_0 \in \R$, we would have
$j = \alpha^T w^j +\alpha_0$ for $j = 0, 1,2,3$ where $w^0 = (0,0), w^1 = (1,0), w^2 = (1, 1)$, and $w^3 = (0,1)$.
This would imply that $\alpha_0 = 0$ (from $j=0$) and $\alpha$ satisfies $\alpha_1 = 1$, $\alpha_1+\alpha_2 = 2$, and $\alpha_2 = 3$ simultaneously, which is not possible.
\end{example}

Let $I=\{1,\ldots,l\}$ and consider a binarization scheme $\B=(B^1,\ldots, B^l)$ defined by affine binarization polytopes.
For each $B^i$ let $x_i= a_i^Tz_i + b_i$ hold where $z_i$ denotes the vector of binary variables associated with $x_i$.  
Furthermore, if all $a_i\in \Z^n, b_i \in \Z$, then 
\begin{equation}
\label{eq:non-original-variables}
\SC(P_\B, I_\B) = \SC(P_\B, I'),
\end{equation}
where $I_\B$, defined in \eqref{def:IB}, contains the indices of the original integer variables as well as the indices of the binarization variables, 
whereas $I'=I_\B\setminus I$ contains the indices the binarization variables only.
To see this, simply substitute each $x_i$ in the inequalities defining a split set $S \in \SS(I_\B)$ by $a_i^Tz_i + b_i$ to obtain an equivalent  split set in $\SS(I')$.
We next observe that \eqref{eq:non-original-variables} does not necessarily hold for non-affine binarization schemes.

\begin{proposition}
  There exists a polyhedral mixed-integer set $P^I$ and a binarization scheme $\B=(B^1,\ldots, B^{|I|})$ composed of non-affine binarization polytopes such that
$$
\SC(P,I) \subsetneq \proj_x(\SC(P_\B, I')), \mbox{ where } I'=I_\B\setminus I.
$$
\end{proposition}
\begin{proof}
Let $P = \{x \in [0,3]^2 :  0 \leq x_2 - x_1 \leq 0.5 \}$. Let $\B = (B, B)$ where $B$ is defined by \eqref{eq:non-linear}. Let $I = \{1,2\}, I' = \{3,4,5,6\}$. 
  By definition, $\SC(P,I) = \SC(P)$, and
  \[P_\B = \{(x,z_1,z_2)\in \R^2 \times \R^2 \times \R^2 : (x_i, z_i) \in B, \mbox{ for } i=1,2\}, \]
  
Clearly, we have $\conv(P \cap \Z^2) = \{x \in [0,3]^2 : x_2 - x_1 = 0\} = \SC(P)$.  Therefore, $\proj_x(\SC(P_\B, I')) \supseteq \SC(P)$.
Let $(\bar x, \bar z) \in P_\B$ be defined by
  \[ \bar x = \left(\begin{array}{c}1 \\ 1.5 \end{array}\right), 
     \bar z = \left(\begin{array}{cc}0.5 & 0.5 \\ 0.5 & 0.5 \end{array}\right). \]
     Clearly $\bar x \notin \SC(P)$. We will show that $(\bar x, \bar z) \in \SC(P_\B, I')$ and thus $\proj_x(\SC(P_\B, I')) \neq \SC(P)$.

     Suppose $(\bar x, \bar z) \not\in \SC(P_\B, I')$. Then there exists a split set $S \in \SS(I')$ such that
     \begin{equation}\label{notinS}
       (\bar x, \bar z) \notin \conv(P_\B \setminus S).
     \end{equation}
     Let $S = \{(x,z) : \delta < a z_{11} + b z_{12} + c z_{21} + d z_{22} < \delta +1\}$ with $a,b,c,d,\delta \in \Z$.
   Then $(\bar x, \bar z) \in S$ which implies that $a \bar z_{11} + b \bar z_{12} + c \bar z_{21} + d \bar z_{22} = \delta + .5$ (as $\bar z$ is half-integral).
     
The points $(\bar x, z')$ and $(\bar x, z'')$ defined by $z' = \bar z + d_z$ and $z'' = \bar z - d_z$ where
  \[ d_z = \left(\begin{array}{cc}0 & 0 \\ 0.5 & 0 \end{array}\right) \]
          are both contained in $P_\B$ and  $(\bar x, \bar z) = .5(x', z') + .5(x'', z'')$.
          If $|c| \geq 1$, then $S$ contains neither $(x',z')$ nor $(x'', z'')$, contradicting (\ref{notinS}); therefore $c = 0$.

The points $(x', z')$ and $(x'', z'')$ defined by $(x',z') = (\bar x, \bar z) + (d_x,d_z)$ and $(x'', z'') = (\bar x, \bar z) - (d_x, d_z)$, where
  \[ d_x = \left(\begin{array}{c}1 \\ 1 \end{array}\right), d_z = \left(\begin{array}{cc}0.5 & 0.5 \\ 0 & 0.5 \end{array}\right), \]
          are both contained in $P_\B$ and $(\bar x, \bar z) = .5(x', z') + .5(x'', z'')$.
          For $z'$, the expression $az_{11}' + bz_{12}' + dz_{22}'$ is integral as $a,b,d$ are integral, $c = 0$, and $z_{11}', z_{12}'$ and $z_{22}'$ are integral.
          Therefore $(x',z') \not\in S$. Similarly, $(x'', z'') \not\in S$, and $x \in \conv(P_\B \setminus S)$, a contradiction.\qed
\end{proof}

We next show that an affine binarization polytope can be transformed into a linear binarization polytope using a unimodular transformation. 
As split closures are invariant under unimodular transformations, this observation implies that there is no additional benefit in using affine binarizations in terms of cutting.
\begin{proposition}
	Let $B \in \Gamma^q_u$ be an affine binarization polytope.  
	Then there exists a linear binarization polytope $B'\in \Gamma^q_u$ that is a  unimodular transformation of $B$. 
\end{proposition}
\begin{proof}
	Let $(0,\bar w) \in B \cap (\R \times \{0,1\}^q)$.  
	Since $B$ is an affine binarization, there exist $\alpha \in \R^q$ and $\alpha_0 \in \R$ such that $B \subseteq \{(x,z) : x = \alpha^T z + \alpha_0\}$.  By definition, we must have $\alpha_0 = - \alpha^T \bar w$.
	
	Now define $f : \R^q \to \R^q$ as
	$$	f(z) = \bar w + Dz, 	$$
	where $D$ is the diagonal matrix with $D_{jj} = (1-2\bar w_j)$. 
	Since $\bar w \in \{0,1\}^q$, $D_{jj} \in \{-1,1\}$, and therefore $D$ is unimodular. Consequently, $f$ is invertible, with 
	$$	f^{-1}(z) = D^{-1}(z - \bar w) = D^{-1}z + \bar w.	$$
	Note that if $y=f(v)$ for $v \in \{0,1\}^q$, then $y_i = 1-v_i$ if $\bar w_i=1$ and  $y_i = v_i$, otherwise. In particular, $D\bar w = -\bar w$ and thus $f(\bar w)= 0$.
	Define $B' = \{(x,f(z)) : (x,z) \in B\}$.  
	As  $f(\{0,1\}^q) = \{0,1\}^q$,  $B'$ is a binarization polytope.  For any $(x,f(z)) \in B'$, we have
	$$
	x = \alpha^T z + \alpha_0 
	= \alpha^T f^{-1}(f(z)) - \alpha^T \bar w
	= \alpha^T (D^{-1}f(z) + \bar w) - \alpha^T \bar w
	= \alpha^T D^{-1} f(z)
	$$
	Thus, $B' \subseteq \{(x,z') : x = \alpha^T D^{-1} z'\}$, and hence $B'$ is a linear binarization.\qed
	
\end{proof}

\subsection{Perfect binarizations}
The set $\Gamma^q_u$, as defined in (\ref{eq:gamma}), contains infinitely many polytopes,
and therefore one can define infinitely many binary extended formulations of $P$ of the form \eqref{eq:pb}.
We next look at a natural finite subset of $\Gamma^q_u$.

A binarization polytope $B\in\Gamma^q_u$ is \emph{exact} if for each $x \in \{0,1,\ldots,u\}$ there is a unique $z \in \{0,1\}^q$ such that $(x,z) \in B$.  Therefore, $B$ is exact provided that $\proj_{x}\colon B \cap (\rr \times \{0,1\}^{q}) \to \{0, \dots, u\}$ is a bijection.  
We say that a binarization polytope $B \in \Gamma^q_u$ is {\em perfect} if it is exact and $B=\conv(B \cap (\R \times \{0,1\}^q))$.   
Thus, if $B$ is perfect, then it is the convex hull of $u+1$ points of the form $(k, w^k)$ for $k=0, \ldots, u$.
As each $w^k\in\{0,1\}^q$ 
there are at most $2^{q(u+1)}$ distinct 
perfect  binarization polytopes in $\Gamma^q_u$.
Also note that if $B\in\Gamma^q_u$, then it has dimension at most $u$.

\newcommand{\thickbar}[1]{\mathbf{\hat{\text{$#1$}}}}
\begin{proposition}
\label{uniqueness}
Consider the extended formulation $P_\B$ of $P$ where $\B=(B^1,\ldots, B^l)$ and $B^i$ is not perfect for some $i$.
If $\thickbar\B$ is obtained from $\B$ by replacing $B^i$ with $\thickbar B^i$ where $\thickbar B^i$ is perfect, then $P_{\thickbar \B} \subseteq P_\B$.
\end{proposition}
\begin{proof}
Without loss of generality, suppose that $B^1\in \Gamma^{q_1}_{u_1}$ is not perfect.  
For each $k \in \{0, \ldots, u_1\}$, choose a corresponding $w_k^1 \in \{0,1\}^{q_1}$ such that $(k, w_k^1) \in B^1$.  If $B^1$ is not exact, this choice will not be unique for some values of $k$.
Let $\thickbar B^1 = \conv( \{ (k, w_k^1) : k=0, \dots, u_1\})$.  Then $\thickbar B^1$ is in $\Gamma^{q_1}_{u_1}$, perfect and contained in $B^1$.  It follows that $P_{\thickbar \B} \subseteq P_{\B}$.\qed
\end{proof}

The following is an example of a binarization polytope that is not exact.  
\begin{example}
Let $P\subset\R^n$ be a polyhedron with $0 \leq x_i \leq 7$ for $i\in I$.
Consider the binarization polytope 
$B = \{ (x,z) \in \R \times [0,1]^4  :  x = 5z_3 + \sum_{j=0}^2 2^j z_{j}\}$ and  
the associated extended formulation 
$$
P_{\B} = \{ (x,z) \in \R^n \times [0,1]^q : x \in P, x_i = 5z_{i3} + \sum_{j=0}^2 2^j z_{ij} \text{ for all } i \in I\},
$$
where $q=4|I|$.
Notice that  $x_i \in \{5,6,7\}$ has two possible representations, one with $z_{i3}=0$ and a second with $z_{i3}=1$.  
Therefore we can define another valid binary extended formulation by setting $z_{i3}$ to zero:
$P_{\B'} = P_{\B} \cap \{(x,z) : z_{i3}=0, ~i\in I\}$.
Since $P_{\B'} \subseteq P_{\B}$, by Equation~\eqref{containment}, we have $SC(P_{\B'}) \subseteq SC(P_\B)$.   
\end{example}


Note that $\BF(u)$ and $\BU(u)$ are perfect binarization polytopes whereas $\BLG(u)$ is exact but not perfect unless $u + 1$ is a power of 2.
By Proposition \ref{uniqueness} perfect binarization polytopes are more desirable as they lead to  stronger extended formulations.
We define the perfect version of $\BLG(u)$ as
$$
\BLGp(u) = \conv(\BLG(u) \cap (\R \times \{0,1\}^q)).
$$
We next show that the binarization polytope $\BLG(u)$ can be made perfect by adding at most $u$ inequalities to $\BLG(u)$.  For this, 
we adapt a result from~\cite[Corollary 2.6]{LS92} about knapsack polytopes with \emph{superincreasing} coefficients.   We give a proof here to explicitly construct the required inequalities in this context.
\begin{proposition}
\label{prop:cropped-cube}
Let $b$ be a positive integer. Let $\bar a \in \rr^n$ such that $\bar a_i = 2^{i-1}$. The binary knapsack polytope $P = \conv(\{ x \in \{0,1\}^n : \bar a^Tx \leq \bar b\})$ can be described by at most $n-1$ inequalities plus the bounds $0 \leq x \leq 1$.  Furthermore, these inequalities can be computed in polynomial time.
\end{proposition}
\begin{proof}
If $\bar b > 2^n-1$, then $P = [0,1]^n$.  Otherwise, 
we compute the binary expansion of $\bar b$ as $\bar b = \sum_{j=1}^n 2^{j-1} \bar x_j$ for $\bar x \in \{0,1\}^n$.
Let $J$ be the set of indices such that $\bar x_j = 0$; $|J| \leq n-1$ as $\bar b > 0$.  For each $j \in J$, define the vector $\bar a_j \in \{0,1\}^n$ as 
$$
a_{jk} = \begin{cases}
\bar x_k & \text{ if } k > j,\\
1 & \text{ if } k = j,\\
0 & \text{ if } k < j,
\end{cases}
$$
where $a_{jk}$ stands for the $k$th coefficient of the vector $a_j$.
Define $b_j = \sum_{j=1}^n a_{jk} - 1$.  The inequalities $a_j^Tx \leq b_j$ are known as cover inequalities for $P$.
Let $A$ be the matrix whose rows are $a_j^T$ for $j \in J$ and define $Q = \{ x \in [0,1]^n : Ax \leq b\}$.
Since $A^T$ has the so-called \emph{consecutive 1's property}, it follows that $A$ is totally unimodular.  Since $b \in \Z^I$, we have that $Q$ is an integral polytope.  

We claim that $P = Q$.
Since $P$ is also an integral polytope, it suffices to show that $P\cap \{0,1\}^n = Q \cap \{0,1\}^n$.   To this end, consider any $x \in \{0,1\}^n$.

Suppose first that $x\notin Q$.  Thus, for some $i \in I$, 
$a_i^T x > b_i$.  Then $x_j > a_{ij}$ for all $j=1, \dots, n$.  Then
$
\bar a^T x \geq \bar a^T a_i > \bar a^T \bar x \geq \bar b.
$
 Thus, $x \notin P$.

Conversely, suppose $x \notin P$.  Since $\bar x \in P$, $x \neq \bar x$.  Let $i$ be the largest index such that $x_i > \bar x_i$.  Then $x_i = 1$ and $\bar x_i = 0$.  But then $a_i^T x > a_i^T \bar x = b_i$.  
Thus,  $x \notin Q$.  

Hence $P\cap \{0,1\}^n = Q \cap \{0,1\}^n$ and  $P = Q$.  Lastly, since $|I| \leq n-1$, $P$ is described by the inequalities $0 \leq x \leq 1$ and at most $n-1$ additional inequalities. \qed
\end{proof}

Combining this result with the fact that 
	$
	\BLGp(u) = (\R \times Q) \cap \{(x,z) : x = \sum_{j=1}^{q} 2^{j-1} z_j\}
	$
where $Q$ is described in the proof of Proposition~\ref{prop:cropped-cube}, we have the following corollary.

\begin{corollary}
$\BLGp(u)$ can be described by one equation and at most $q-1$ inequalities for $q = \ceil{\log_2(u+1)}$ and the simple bound constraints.  These inequalities can be computed in polynomial time.
\end{corollary}

\subsection{The logarithmic binarization is better than the original formulation}\label{sec-log}

We next give an example for which the projection of the split closure of the logarithmic extended formulation is strictly contained in the split closure of the original formulation.
Let
\[P=\Big\{x\in[0,2]^2\::\:\textstyle2x_1+x_2\le5,~-2x_1+3x_2\le 3\Big\},\]
and the associated integer set $P^I\subset\Z^2$ where $I=\{1,2\}$. 
Now consider the extended formulation of $P$ obtained by using the logarithmic binarization scheme:
\[\PLG=\Big\{(x,z)\in\R^{2}\times[0,1]^4\::\:x\in P,~x_i = z_{i1} + 2z_{i2},~\text{ for }i=1,2\Big\}.\]
Note that the logarithmic binarization polytope $\BLG(u)$ is not perfect for $u=2$.
\begin{figure}[h]\begin{center}\begin{tikzpicture}[scale=1.5] 
\draw[dashed, , opacity=0.5] (0,0) -- (2,0) -- (2,2) -- (0,2) -- (0,0);
\draw [ ->](0,0) -- (0,2.2); \draw [ ->](0,0) -- (2.5,0); 

\draw [fill = black, opacity=0.1] (0,0) -- (2,0) -- (2,1) -- (1.5,2) -- (0,1) -- (0,0);
\draw  (0,0) -- (2,0) -- (2,1) -- (1.5,2) -- (0,1) -- (0,0);

\foreach \x in {0,1,2}\foreach \y in {0,1,2}
\draw[fill = white] (\x,\y) circle (.75pt);

\foreach \x in {0,1,2} \draw (\x cm,.1pt) -- (\x cm,-.1pt) node[anchor=north] {$\x$};
\foreach \y in {0,1,2} \draw (.1pt,\y cm) -- (-.1pt,\y cm) node[anchor=east] {$\y$};

\draw[fill = black] (1.5,2) circle (.65pt); \node at (1.5,2.1) {$p_2$};
\draw[fill = black] (1,5/3) circle (.65pt);\node at (.9,1.8) {$p_1$};
\draw[fill = black] (5/4,3/2) circle (.65pt); \node at (1.3,1.3) {$\bar x$};
\draw[fill = black] (5/3,5/3) circle (.65pt); \node at (1.8,1.8) {$p_3$};

\draw[dashed] (1,5/3) -- (2,1);\draw[dashed] (1,1) -- (1.5,2); 

\draw[dashed]  (5/3,5/3)-- (0,1); 
\end{tikzpicture}\end{center}
\caption{Polytope $P$}\label{fig:proofbypicture}
\end{figure}

\begin{theorem}
  For $P$ and $\PLG$ defined above, we have $\projk{x}{\SC(\PLG)} \subsetneq \SC(P)$.
\end{theorem}
\begin{proof}
We will show that the point $\bar x=(1.25,1.5)\in P$ is contained in the split closure of  $P$
 but not in the projection of the split closure of the associated logarithmic extended formulation $\PLG$.

Suppose $\bar x\not\in SC(P)$. 
Then  $\bar x\not\in \conv(P \setminus S)$ for some split set 
$S = \{x \in \R^2 : \pi_0 < \pi^Tx < \pi_0+1\}$ where $\pi, \pi_0$ are integral.
By definition, $P \setminus S = P \cap (S_1 \cup S_2)$ where $S_1 = \{x \in \R^2 : \pi^Tx \leq \pi_0\}$ and $S_2 = \{x \in \R^2: \pi^Tx \geq \pi_0 +1\}$ and $\Z^2 \subseteq S_1 \cup S_2$.
Note that $\bar x$ lies in the convex hull of $(2,1) \in P$ and $p_1=(1,5/3)\in P$ and therefore, $p_1\in S$.
In a similar manner, we can conclude that $p_2=(1.5,2) \in P$ and $p_3=(5/3,5/3)\in P$ are both contained in $S$
as $\bar x$ lies in the convex hull of $p_2$ and $(1,1) \in P$, and also in the convex hull of $p_3$ and $(0,1) \in P$.
See Figure~\ref{fig:proofbypicture}.

Consequently $(1,1)$ and $(1,2)$ are not contained in the same $S_i$, otherwise $p_1$ (a convex combination of the previous two points) would be contained in the same $S_i$, contradicting $p_1 \in S$.
Similarly, we can conclude that $(1,2)$ and $(2,2)$ are not in the same $S_i$ (otherwise $p_2$ would be contained in the same $S_i$),
and $(1,1)$ and $(2,2)$ are not in the same $S_i$ (otherwise $p_3$ would be contained in the same $S_i$).
Given that there are only two choices for $S_i$, we get a contradiction.

We next show that the inequality $x_2\le1.4$ is valid for the split closure of $\PLG$.
  We will first argue that the following inequalities are split cuts for $\PLG$:
\begin{eqnarray} 
   z_{21}  	+ \phantom{3}	z_{22} &\leq& 1  	\label{gc1}\\
   z_{22}  	- \phantom{3}	z_{12} &\leq& 0  	\label{s2}\\
    	- 2 			z_{11} + 3x_2 &\leq& 3 	\label{s3} \\
   2x_1  	+ 3 			x_2 &\leq& 7. 	\label{s4} 
  \end{eqnarray}
The inequality  (\ref{gc1}) can be obtained as a Gomory-Chv\'atal cut from the inequality  $2z_{21}+ {2}z_{22} \leq 3$ which is implied by  $z_{21}+ {2}z_{22} = x_2 \leq 2$ and $z_{21} \leq 1$.
To obtain inequality  (\ref{s2}), replace $x_i$ in $-2x_1 + 3x_2 \leq 5$ by $x_i = z_{i1} + 2z_{i2}$ to get $-2z_{11} -4z_{12} + 3z_{21} + 6z_{22} \leq 5$. Adding the valid inequalities $2z_{11} \leq 2$, $-3z_{21} \leq 0$ and $-2z_{22} \leq 0$ for $\PLG$ to the previous inequality, we get $4z_{22} - 4z_{12} \leq 7$ which yields (\ref{s2}) as a Gomory-Chv\'atal cut for $\PLG$.

To see that inequality  (\ref{s3}) is a  split cut, consider the disjunction $z_{12}\le0$ or  $z_{12}\ge1$.
If $z_{12}\ge1$, then $z_{11} = 0$ and $x_1=2$, and therefore $x_2\le1$ as $2x_1+x_2\le5$ for all $x\in P$ and (\ref{s3}) is satisfied.
On the other hand, if $z_{12}\le0$, then $x_1=z_{11}$ and 	(\ref{s3}) is implied by the second inequality defining $P$, namely, $-2x_1+3x_2\le 3$.
Finally,  inequality  (\ref{s4}) can be obtained as a split cut from the disjunction $x_1\le1$ or $x_1\ge2$.
If $x_1\le1$, then the inequality $-2x_1+3x_2\le 3$ defining $P$ implies $3x_2 \leq 5$ and (\ref{s4}) is satisfied. On the other hand, if $x_1=2$, then the inequality $2x_1+x_2\le5$ defining $P$ implies that $x_2\le1$ and therefore  (\ref{s4})  holds.

Combining inequalities (\ref{gc1})-(\ref{s4}) with the multipliers 4,4,1, and 1, respectively, gives the following inequality:
$$
2 x_1  + 6 x_2 -2 z_{11} -4 z_{12}+ 4 z_{21}  + 8 z_{22} ~\le~ 14.
$$
As $-2 z_{11} -4 z_{12}= -2x_1$ and $4 z_{21}  + 8 z_{22} =4x_2$, this simplifies to $10 x_2\le 14$ and therefore, $x_2\le 1.4$ is indeed valid for the split closure of $\PLG$.\qed
\end{proof}


\ignore{
\begin{figure}
\begin{center}
\includegraphics[scale = 0.3]{domination}
\input{dominance_relations.tikz}
\end{center}
\caption{\tred{Possible graphic to put here}}
\end{figure}
}
\subsection{Strength of single disjunctions in extended space}\label{sec-onedisj}

%
%
%

The following result shows that if a split disjunction in the extended space only involves the binarization variables associated with a single original variable, then it is not more useful than a split disjunction involving the original variable itself.

\begin{proposition}
\label{prop:individual-split}
Let $P^I$ be a given polyhedral mixed-integer set, with $I = \{1, \ldots, l\}$ and let $\B = (B^1, \ldots, B^l)$ a binarization scheme.
Let $z_1$ be the auxiliary binary variables associated with $x_1$.
Then, for any split set $S = \{ (x,z) \in  \R^{n+q} :  \pi_0 < \pi^T z_1 < \pi_0 + 1\}$ in the extended space, there exists a split set $S'$
\ignore{ = \{ x\in \R^n : \pi_0' <  x_1 < \pi_0 ' +1\}} 
in the original space such that 
$$
\proj_x(P_\B \setminus S) \supseteq P\setminus S'
$$
\end{proposition}
\begin{proof}
Let $B^1 \in \Gamma^{q_1}_u$.
$$P_\B \setminus S=P_\B\cap\{ (x,z) \in  \R^{n+q} : (x_1, z_1)\in  A_0 \cup A_1\},$$
where  $A_0 = \{ (x_1, z_1) \in B : \pi^T z_1 \leq \pi_0 \}$ and $A_1 = \{ (x_1, z_1) \in B : \pi^T z_1 \geq \pi_0 + 1\}$. 
As $B^1$ is a binarization polytope, there exists a point $p^t=(t,w^t)\in  B^1$ for each $t\in \{0, \dots, u\}$ such that $w^t\in \{0,1\}^{q_1}$.  Without loss of generality, assume that $p^0\in A_0$.  Let $s\in\{0,\ldots,u\}$ be the largest index such that $p^s\in A_0$.  We claim that choosing $\pi_0' = s$ is sufficient.

If $P \subseteq \proj_x(P_\B \setminus S)$, then the result holds trivially.  Thus, suppose this is not the case and consider any $\hat x \in P \setminus  \proj_x(P_\B \setminus S)$.  
Consider any distinct $r,t\in\{0,\ldots, u\}$ with  $t\ge \hat x_1\ge r$.  
Define $\hat p=(\hat x_1,\hat w)=\lambda p^t + (1-\lambda) p^r$ where  $\lambda = (\hat x_1-r)/(t-r)$.
Note that $ \lambda \in [0,1]$ and therefore  $\hat p\in B^1$.
As $P_\B$ is an extended formulation of $ P$ and $\hat x\in P$, there exists a point $(\hat x, \hat z) \in P_\B$.  
Moreover, as $\hat p\in B^1$,  we also have  $(\hat x,  z') \in P_\B$ where  $ z' = (\hat w, \hat z_2, \dots, \hat z_l)$ with $l = |I|$.

If $p^r,p^t\in A_0$, then, as $A_0$ is convex, we  have $\hat p \in A_0$.
Therefore $\hat x \in \proj_x(\{(x,z) \in P_\B : (x_1,z_1) \in A_0\}) \subseteq \proj_x(P_\B \setminus S)$.  
A similar argument holds if instead $p^r,p^t\in A_1$.  Therefore, if $p^r,p^t\in A_0$ or $p^r,p^t\in A_1$, then 
$\hat x \in \proj_x(P_\B \setminus S)$.

Thus, since we assumed that $p^0 \in A_0$, for all pairs $r,t \in \{0,\dots, u\}$ with $r\le \hat x_1\le t$, we have $p^r \in A_0$ and $p^t \in A_1$.  Therefore $p^0, \dots, p^s \in A_0$, $p^{s+1} \in A_1$ and $s <  \hat x_1 < s+1$.  Since $\hat x$ was chosen arbitrarily in $P \setminus  \proj_x(P_\B \setminus S)$, this shows that $P \setminus  \proj_x(P_\B \setminus S) \subseteq \{ x \in P :  s< x_1 < s+1\}$.  
\qed
\end{proof}

Also note that this result also implies that 
$$\proj_x(\conv(P_\B \setminus S)) \supseteq \conv(P\setminus S')$$
and we observe that split disjunctions in the extended space must involve  binarization variables associated with multiple original variables in order to generate cuts that cannot be obtained using original variables.
We next give an example where cuts from a single split set in the extended space can give the convex hull of a mixed-integer set while there is no split set, or more generally, no lattice-free convex set in the original space that can do the same.

\begin{example}
  Let $P = \{x \in [0,2]^2  : x_2 = \tfrac{1}{2} x_1 + \tfrac{1}{2}\}$ and $I=\{1,2\}$.
  Then $P^I$ consists of a single point $p=(1,1)$.
As $p$ is contained in the relative interior of $P$, there is no lattice-free convex set (e.g., a split set) $S\subseteq\R^2$ that satisfies $P^I=\conv(P\setminus S)$. Let $\B = (B_1,B_2)$, where $B_i = B^{\FF}(2)$ is the full binarization polytope \eqref{bin-sa} with $u=2$ for $i=1,2$, i.e., 
$$B_i =  \{(x_i,z_i)\in\R\times[0,1]^{2}\::\:\textstyle x_i=z_{i1}+2z_{i2},~z_{i1}+z_{i2}\le1\}.$$
Let $P_\B$ be the binary extended formulation of $P$ defined by $\B$.
For $(x,z) \in P_\B$,
\begin{equation}\label{eq:x2z12}
x_2 - z_{12} = \frac{1}{2} x_1+ \frac{1}{2} - z_{12} = \frac{1}{2} (z_{11} + 2z_{12}) + \frac{1}{2} - z_{12} = \frac{1}{2}(z_{11} + 1)> 0.
\end{equation}
Let $S = \{(x,z_1, z_2) : 0 < x_2 - z_{12} < 1\}$ be a split set in the space of $P_\B$.
Then $P_\B\setminus S$ consists of points in $P_\B$
that satisfy  $x_2-z_{12} \leq 0$ or $x_2-z_{12} \geq 1$.
Because of (\ref{eq:x2z12}), there are are no points in $P_\B$ that satisfy the first inequality, and all points in $P_\B$ satisfying $x_2 - z_{12} \geq 1$
also satisfy $(z_{11} + 1)\ge2$ and thus the equations $z_{11}= 1, z_{12} = 0$, $x_1 = 1$, and $x_2 = 1$. 
Therefore 
$$
\proj_x( \conv(P_\B \setminus S)) = \{ (1,1)\} = P^I.
$$
\end{example}

Bonami and Margot~\cite{BM} have already observed that the rank-2 simple split closure of the unary  binarization \eqref{bin-roy} always leads to the integer hull in the original space when the original set only has two integer variables.
The example above shows that in some cases this might  happen even with a single split cut in the extended space when the split disjunction combines binarization variables associated with different original variables.

\section{Relative strength of binarization schemes}\label{sec:compare}
We next compare various binary extended formulations with respect to the strength of the projection of their split closures.
Our main result implies that the full and unary-binarization schemes lead to the strongest extended formulations. 
We then give a hierarchy of other schemes considered earlier in the paper.

%
\subsection{Strength of unimodular binarization schemes}
We next characterize a  class of binarization schemes that have equally strong projected split closures.

\begin{definition} Let $B \in \Gamma^u_u$ (for some $u > 0$) be a perfect binarization polytope, i.e., there exist 0-1 vectors $w^0, \ldots, w^u \in \{0,1\}^u$ such that $B$ is the convex hull of the points $(i, w^i)$.
	We say that $B$ is {\em unimodular} if  the $u \times u$ matrix with columns $w^j - w^0$ for $j=1, \ldots, u$ is unimodular.
\end{definition}

Recall that the full binarization polytope $\BF(u)$ and the unary binarization polytope  $\BU(u)$ are perfect. 
Moreover,  $\BF(u)$ is equal to the convex hull of points $(j,e^j)$, where $e^j$ is the $j$th standard unit vector for $j=1, \ldots, u$, and $e^0$ is the all-zeros vector.
Similarly, $\BU(u)$ is equal to the convex hull of points $(j,d^j)$, where $d^j=\sum_{i=0}^j e^i$  for $j=0, \ldots, u$.
Consequently, both these polytopes are unimodular. 

We next present some technical results that we need for the main result.
We start off by generalizing a result in \cite[Proposition 3]{dgl} on unimodular transformations to integral, affine transformations. 
\begin{theorem}\label{thm:trans}
	Let $P \subseteq \rr^m$, $Q \subseteq \rr^n$,  $I = \{1,\ldots, \ell\}$, and $I' = \{1,\ldots, \ell'\}$ where $\ell\le m$ and $\ell'\le n$.
	Let $f(x) = (g(x^1), x^2)$ where $x= (x^1, x^2)$ and $g \colon \R^{\ell} \to \R^{\ell'}$ is  an integral, affine transformation. 
	If $f(P) \subseteq Q$, then for any integer $k \geq 1$, $$f(\SC^k(P, I)) \subseteq SC^k(Q, I').$$
\end{theorem}
\begin{proof}
	
	Let $\mathcal S$ be a collection of sets in $\R^n$.
	For any $S \in \mathcal{S}$, if $x \in S $ then $f(x) \in f(S)$, and therefore
	\beqn \textstyle f\left(\bigcap_{S \in \mathcal{S}} S\right) \subseteq \bigcap_{S \in \mathcal{S}} f(S).\label{eq:fsub}\eeqn
	Furthermore, note that $f(\sum_{i=1}^t \lambda_i x_i) = \sum_{i=1}^t\lambda_if(x_i)$ for any $x_1, \ldots, x_t \in \R^m$ and any $\lambda_1, \ldots, \lambda_t\in \R$ satisfying $\sum_{i=1}^t \lambda_i = 1$.
	Therefore, for any $T \subseteq \rr^m$
	\begin{equation}\label{conv-commute}
	f (\conv(T))= \conv(f(T)).
	\end{equation}
	In addition,  
	$$
	f(P) \setminus S = \{ f(x) : x \in P, f(x) \notin S\} = \{ f(x) : x \in P, x \notin f^{-1}(S)\} = f(P \setminus f^{-1}(S))
	$$
	Taking $T = P \setminus f^{-1}(S)$ in (\ref{conv-commute}), we see that
	\beqn f( \conv(P \setminus f^{-1}(S))) = \conv(f(P)\setminus S).\label{eq:fconv}\eeqn

	Let $g(x_1) = Vx_1 + v$ where $V \in \Z^{\ell' \times \ell}$ and $v \in \Z^{\ell'}$.
	Consider the split set  $S\in\SS_n(I')$ given by
	$S = \{(y_1,y_2) \in \R^{\ell'} \times \R^{n - \ell'} : \pi_0 < \pi_1^Ty_1 + \pi_2^Ty_2 < \pi_0+1\},$
	where $\pi_1, \pi_2$ and $\pi_0$ are integral and $\pi_2 = 0$.
	Then
	\beqa
	f^{-1}(S) &=& \{(x_1,x_2) \in \R^\ell \times \R^{m-\ell}: \pi_0 < \pi_1^T(Vx_1 + v) < \pi_0 + 1\} \\
	&=& \{(x_1,x_2) \in \R^\ell \times \R^{m-\ell}: \pi_0 - \pi_1^Tv < \pi_1^TVx_1 < \pi_0 + 1 - \pi_1^Tv\}.
	\eeqa
	As $\pi_1^TV$  and  $\pi_1^Tv$ are integral, we see that $ f^{-1}(S)$ is a split set in $\SS_m(I)$.
	Therefore, $\{f^{-1}(S) : S \in \SS_n(I')\} \subseteq \SS_m(I)$, and
	\begin{equation}\label{eq:subs} \SC(P, I) = \bigcap_{S \in \SS_m(I)}\conv(P \setminus S)
	\subseteq \bigcap_{S \in \SS_n(I')}\conv(P \setminus f^{-1}(S)).
	\end{equation}
	Then 
	\begin{eqnarray*} f(\SC(P, I)) &\subseteq& f(\bigcap_{S \in \SS_n(I')}\conv(P \setminus f^{-1}(S))) \\
		&\subseteq& \bigcap_{S \in \SS_n(I')} f( \conv(P \setminus f^{-1}(S)))\nonumber\\
		&=& \bigcap_{S \in \SS_n(I')}\conv(f(P) \setminus S) ~\subseteq~  SC(Q,I'),
	\end{eqnarray*}
	where the first inclusion follows from (\ref{eq:subs}) and the second one follows from \eqref{eq:fsub}. The next equality follows from \eqref{eq:fconv} and the final inclusion follows from \eqref{containment} and the fact that $P\seq Q$.
	
	Therefore the claim holds for $k=1$ and the result follows by induction on $k$.\qed
\end{proof}

\begin{lemma}\label{lem:fexists}
	Let $q, u$ be positive integers and  let $B \in \Gamma^u_u$ and $C \in \Gamma^q_u$. 
	If $B$ is unimodular binarization polytope then there exists an integral, affine transformation of $B$ into $C$. 
\end{lemma}
\begin{proof}
	As $B$ is perfect, $B=\conv(\{(j, v^j): j=0, \ldots, u\})$ for some  $v^j\in\{0,1\}^u$ and  $C$ contains points $(j,w^j)$ for some  $w^j\in\{0,1\}^q$ for all $j=0, \ldots, u$. 
	Let $V$ be the $u\times u$ unimodular matrix with columns $v^j-v^0$ and let $W$ be the integral matrix with columns $w^j-w^0$.

	Define the integral affine transformation $f : \R^u \rightarrow \R^q$ as $f(z) = WV^{-1}z-WV^{-1}v^0+w^0$ and note that $WV^{-1}$ is an integral matrix and $v^0,w^0$ are integral vectors. Furthermore, 
	\beqa f(v^j) &=& WV^{-1}v^j-WV^{-1}v^0+w^0~=~WV^{-1}(v^j-v^0)+w^0.\eeqa
	As $WV^{-1}V=W$, we have $f(v^j)=w^j$.
	
	In addition, let $g(x,z)  = (x, f(z))$ and note that  $g : \R^{u+1} \rightarrow \R^{q+1}$ is also an integral affine transformation.
	As $g$ is affine, it commutes with the convex hull operator $\conv(\cdot)$ and
	\begin{align*}
	g(B) = g(\conv(\{(j, v^j): j=0, \ldots, u\})) &= \conv(\{ g(j, v^j) : j=0, \ldots, u\})\\
	&= \conv(\{(j, w^j) : j=0, \ldots, u\}) \subseteq C.
	\end{align*}\qed
\end{proof}

We now prove the main result of this section.
\begin{theorem}\label{thm:main}
  Let $P$ be defined as in (\ref{eq:p}), and let $I = \{1, \ldots, l\}$.
 Consider a binarization scheme $\B= (B^1, \ldots, B^l)$  where each $B^i$ is unimodular  and let $\C= (C^1, \ldots, C^l)$ be an arbitrary binarization scheme.
  Then  for all integers $k\ge1$, 
  $$\projk{x}{\SC^k(P_\B, I_\B)} \subseteq \projk{x}{\SC^k(P_\C, I_\C)}.$$  
\end{theorem}
\begin{proof}
   Lemma \ref{lem:fexists} implies that for each $i=1, \ldots, l$, there exists an integral affine transformation $f_i$ such that the transformation $(x,z)  \rightarrow (x,f_i(z))$ is integral, affine, and maps $B^i$ into $C^i$.
  Therefore, if $( x,  z_1, \ldots,  z_l) \in P_\B$, then $( x_i,  z_i) \in B^i$ and $( x_i, f_i( z_i)) \in C^i$.
    Let $g$ be the integral affine function from the space of $P_\B$ to the space of $P_\C$ defined as follows:
    \[ (x, z_1, \ldots, z_l) \in P_\B \Rightarrow g(x, z_1, \ldots, z_l) = (x, f_1(z_1), \ldots, f_l(z_l)). \]
    Then $g(P_\B) \subseteq P_\C$ and 
    Theorem~\ref{thm:trans} implies that $g(\SC^k(P_\B, I_\B)) \subseteq \SC^k(P_\C, I_\C)$ for all $k \geq 1$.

    Let $\bar x \in \projk{x}{\SC^k(P_\B, I_\B)}$ for some $k \geq 1$. By definition, there exists vectors $\bar z_1, \ldots, \bar z_l$
    such that $(\bar x, \bar z_1, \ldots, \bar z_l) \in \SC^k(P_\B, I_\B)$.
    Therefore
    \[ g(\bar x, \bar z_1, \ldots, \bar z_l) = (\bar x, f_1(\bar z_1), \ldots, f_l(\bar z_l)) \in \SC^k(P_\C, I_\C). \]
    This implies that $\bar x \in \projk{x}{\SC^k(P_\C, I_\C)}$, and the proof is complete.\qed
\end{proof}

The following is a consequence of Theorem~\ref{thm:main}.
\begin{corollary}\label{cor:uni}
	Let $P$ be defined as in (\ref{eq:p}), and let $I = \{1, \ldots, l\}$. If $\B$ and $\C$ are two binarization schemes defined by unimodular binarization polytopes, then $$\proj_x(\SC(P_\B, I_\B)) = \proj_x(\SC(P_\C, I_\C)).$$  
\end{corollary}
In particular we conclude that full and unary binarization schemes are stronger than all other binarization schemes in the sense that the projection of their split closures are equal to each other and are contained in all other projected split closures. 
Moreover, the proof of Theorem \ref{thm:main} implies that the unimodular transformation that maps one unimodular binarization scheme to another also maps its  split closure (in the extended space) to the split closure of the other.


As an other application of Theorem~\ref{thm:main}, let $\B$ be a  binarization scheme defined by unimodular binarization polytope and  let $\C $ be the logarithmic binarization scheme. Therefore we have  $\projk{x}{\SC^k(P_\B, I_\B)} \subseteq \projk{x}{\SC^k(P_\C, I_\C)}$ for all $k \geq 1$.
Furthermore, as $P_\C^{I_\C}$ is defined by $q= \sum_{i=1}^l \ceil{\log_2(u_i+1)}$ binary variables, all vertices of $\SC^q(P_\C, I_\C)$ have integral $z$ values by a result of Balas \cite{B79} on disjunctive cuts.
Therefore they also have integral coordinates for the variables $x_1, \ldots, x_l$. 
Consequently, $\projk{x}{\SC^q(P_\B, I_\B)} = \conv(P^I)$ and we have the following observation.

\begin{corollary}
	Let $P$, $I$ and $\B$ be defined as in Theorem~\ref{thm:main}. 
	Then $\projk{x}{\SC^q(P_\B, I_\B)} = \conv(P^I)$ where $q = \sum_{i=1}^l \ceil{\log_2(u_i+1)}$.
\end{corollary}

\subsection{Perfect logarithmic binarization is better than logarithmic  binarization}\label{sec-nonperf}

We next give an example for which the projection of the split closure of the perfect logarithmic extended formulation is strictly contained in the projection of the split closure of the logarithmic formulation.
Consider
\[P=\{x\in[0,2]^2\::\:\textstyle x_1+10x_2\le20,~ 10x_1+x_2\le 20\},\]
and the associated integer set $P^I\subset\Z^2$ where $I=\{1,2\}$. 
Now consider the extended formulation of $P$ obtained by using the logarithmic binarization scheme:
\[\PLG=\{(x,z)\in\R^{2}\times[0,1]^4\::\:x\in P,~x_i = z_{i1} + 2z_{i2},~\text{ for }i=1,2\},\]
and the extended formulation of $P$ obtained by using the perfect logarithmic binarization scheme:
\[\PLGp=\{(x,z)\in\R^{2}\times[0,1]^4\::\:x\in P,~x_i = z_{i1} + 2z_{i2},~z_{i1} + z_{i2} \leq 1,~\text{ for }i=1,2\}.\]

\begin{theorem}\label{thm:LG}
	For $P$ defined above, we have $\projk{x}{\SC(\PLGp)} \subsetneq \projk{x}{\SC(\PLG)}$.
\end{theorem}
\begin{proof}
We will show that the point $\bar x=(6/5, 6/5)$ belongs to $\projk{x}{\SC(\PLG)}$ but not to  $\projk{x}{\SC(\PLGp)}$.
We will first argue that the following inequalities are split cuts for $\PLGp$:
\begin{eqnarray} 
   z_{11} + z_{12}~~~~~~~~ + z_{22} &\leq& 1 \label{lgc1}\\
   \phantom{ z_{11} +} z_{12} + z_{21} + z_{22} &\leq& 1 \label{lgc2}
\end{eqnarray}
To see \eqref{lgc1} is a split cut, consider the disjunction $z_{22}\le0$ or $z_{22}\ge1$.
When $z_{22}\le0$, \eqref{lgc1} holds as $z_{11} + z_{12}\le1$ is valid for $\PLGp$. 
On the other hand, if $z_{22}\ge1$, then $x_2=2$ and $x_1=0$. Consequently, $z_{11} + z_{12}=0$ and the inequality 
\eqref{lgc1}  holds.
The argument for \eqref{lgc2} is similar using the disjunction $z_{12}\le0$ or $z_{12}\ge1$.

Adding inequalities  \eqref{lgc1}  and  \eqref{lgc2}, we get
$ z_{11} + 2z_{12} + z_{21} + 2z_{22} \leq 2, $
which is the same as $x_1 + x_2 \leq 2$ and therefore  $\bar x\not\in\projk{x}{\SC(\PLGp)}$. 

The proof of the fact that $\bar x=(6/5, 6/5)\in\projk{x}{\SC(\PLG)}$ is in the Appendix.\qed
\end{proof}

\subsection{Unary binarization is better than perfect logarithmic binarization}\label{sec-uni-log}

We next give an example for which the projection of the split closure of the unary extended formulation is strictly contained in the projection of the split closure of the perfect logarithmic extended formulation.
Consider
\beq P = \Big\{(x,y)\in[0,3]^3\times[0,1]^3\::\:\textstyle\sum_{i=1}^3x_i=4,~x_i\le4y_i,~\text{ for }i=1,2,3\Big\}, \label{ex-lu}\eeq
and the associated integer set $P^I\subset\Z^6$ where $I=\{1,2,3,4,5,6\}$. 
Now consider the unary extended formulation of $P$:
\begin{eqnarray*}
  P_{U} = \Big\{(x,y,z)\in\R^{3+3+3\times 3}\::\:(x,y)\in P, && x_i = z_{i1} + 2z_{i2} + 3z_{i3},\\
  && z_{i1} + z_{i2} + z_{i3} \leq 1~\text{ for }i=1,2,3\Big\}, \label{ex-lu-pu}
\end{eqnarray*}
and the perfect logarithmic extended formulation of $P$:
\beq \PLGp = \Big\{(x,y,z)\in\R^{3+3+3 \times 2}\::\:(x,y)\in P,~x_i = z_{i1} + 2z_{i2},~\text{ for }i=1,2,3\Big\}. \label{ex-lu-pl}\eeq
Note that  the logarithmic binarization polytope  $\BLG(u)$ is  perfect for $u=3$.

\begin{theorem}\label{thm:U}
	For  $P$ defined above, $\projk{x,y}{\SC(P_{U})} \subsetneq \projk{x,y}{\SC(\PLGp)}$.
\end{theorem}
\begin{proof}
Let $(\bar x,\bar y) = [(1.5,1,1.5),(.5,.5,.5)]$.
We will argue that the point $(\bar x,\bar y) $ belongs to $\projk{x,y}{SC(\PLGp)}$ but not to  $\projk{x,y}{\SC(P_{U})}$.
First we will show  that the following inequalities are Gomory-Chv\'atal cuts for $P_{U}$:
  \begin{eqnarray}
  y_1 - z_{11} - z_{12} - z_{13} &\geq& 0,  \label{ugc1}\\
  y_2 - z_{21} - z_{22} - z_{23}  &\geq& 0,  \label{ugc2}\\
  y_3 + z_{11} + z_{12} + z_{13} + z_{21} + z_{22} + z_{23} &\geq& 2. \label{ugc3}
  \end{eqnarray}
 To derive inequality (\ref{ugc1}), we take the combination of constraints
 \[ (4y_1 - x_1 \geq 0) + (x_1 - z_{11} - 2 z_{12} - 3 z_{13} = 0) - 3(z_{11} + z_{12} + z_{13} \leq  1) + \]
 \[  (z_{12} \geq 0) + 2(z_{12} \geq 0) \]
to obtain $4y_1 -4z_{11} - 4z_{12} - 4z_{13} \geq -3$ as a valid inequality for $P_{U}$.
Dividing this inequality by 4 and rounding up the resulting right-hand-side,
we obtain (\ref{ugc1}) as a Gomory-Chv\'atal cut for $P_{U}$.
We can obtain (\ref{ugc2}) in a similar manner by taking constraints involving $y_2, x_2, z_{21}, z_{22}, z_{23}$.
Taking the combination of constraints
\[-\frac{1}{12}(x_3 \leq 3) +\frac13(x_1 + x_2 + x_3 = 4) + \frac14(4y_3 - x_3 \geq 0) + \]
\[\frac13(-x_1 + z_{11} + 2 z_{12} + 3 z_{13} = 0) + \frac13(-x_2 + z_{21} + 2 z_{22} + 3 z_{23} = 0). \]
and rounding up the nonzero coeffients of the variables and rounding up the right-hand-side,
we obtain (\ref{ugc3}) as a Gomory-Chv\'atal cut for $P_{U}$.

Adding Inequalities \eqref{ugc1}-\eqref{ugc3}, we obtain $y_1+y_2+y_3\ge2$ is a valid inequality for $\SC(P_{U})$ which is violated by $(\bar x,\bar y)$.
The proof of the fact that  $(\bar x,\bar y) \in\projk{x,y}{SC(\PLGp)}$ is in the Appendix.\qed
\end{proof}

\newcommand{\bb}{B\&B~}
\newcommand{\bbt}{\bb tree}
\newcommand{\TTp}{\TT'}

\section{Branching}\label{sec:branch}
We  next consider binarization in the context of branch and bound (B\&B) trees for  integer programs.   
We will construct a polyhedral set such that the description of its integer hull can be obtained with a much smaller tree when a binary extended formulation is used instead of the original formulation.

%
%

\newcommand{\leaf}{\mathrm{leaf}}
To simplify notation, we will consider a pure-integer set $P^I=P\cap \Z^n$ where $P\subseteq\R^n$ is a polyhedron and $I=\{1,\ldots,n\}$.
A \bb tree for $P^I$ is a rooted binary tree where each node  has either zero or 2 successor nodes.
Nodes in the tree without successor nodes are called leaf nodes and the only node without a predecessor is called the root node. 
For bounded $P \subseteq [0,u]^n$, we label the root node with $D = [0,u]^n$, and similarly, for a binary extended formulation $P_\B \subseteq [0,u]^n \times [0,1]^q$, we  label the root node with a subset of $[0,u]^n \times [0,1]^q$. 
The labels of the non-root nodes  are are obtained from their parent node via ``branching".
More precisely, if a node is labeled with a polyhedron $D'$, its successor nodes are labeled with $D'\cap L$ and $D'\cap R$ where $L=\{y\in\R^n:y_i\le t\}$ and  $R=\{y\in\R^n:y_i\ge t+1\}$ for some  variable $y_i$ where $i\in I$, and $t\in\Z$. 
We refer to $D'\cap L$ as the left successor of $D'$ and  $D'\cap R$ as the right successor.

Let $\TT$ be a \bbt~ for $P^I$, and let $\leaf(\TT)$ denote the labels of the leaf nodes of $\TT$.
From now on we will refer to a node by its label.
Note that 
$$P^I\seq \bigcup_{N \in \leaf(\TT)} N \cap P\seq  P, ~~\text{and}~~\conv (P^I)\seq \conv\Big(\bigcup_{N \in \leaf(\TT)} N \cap P\Big)\seq  P.$$
We will call $\TT$ a \emph{complete} \bbt~with respect to $P^I$ if optimizing any linear function over $P^I$ is the same as optimizing it over $P$ intersected with the leaf nodes of $\TT$. 
In other words, $\TT$ is called complete if
\begin{equation}
\label{eq:complete}
\conv (P^I)= \conv\Big(\bigcup_{N \in \leaf(\TT)} N \cap P\Big).
\end{equation}
 
In an earlier paper, Owen and Mehrotra~\cite{OM01} studied the binary extended formulation $P_\mathcal{B}$  using the full binarization scheme as defined in equations \eqref{bin-sa}.  
They argue that given a \bb tree $\TT_\B$ for $P_\mathcal B^{I_\mathcal B}$, one can construct a \bb tree $\TT$ for $P^I$ with the same number of leaves such that 
\begin{equation}
\label{eq:complete-extended}
\bigcup_{N \in \leaf(\TT)} (N \cap P) \subseteq \bigcup_{N \in \leaf(\TT_\B)} \proj_x(N \cap P_\mathcal{B}).
\end{equation}
They also prove a similar result for the  logarithmic binarization scheme   \eqref{bin-log}.
Thus, it seems that there is no benefit in branching on the auxiliary binary variables and they conclude  that {\em ``remodeling of mixed-integer programs by binary variables should be avoided in practice unless special techniques are used to handle these variables.''}

We also point out that equation \eqref{eq:complete-extended} holds for the unary binarization scheme as well.
To see this, first note that for the unary binarization scheme
$$ (z_{it}\le 0) ~\Longrightarrow~ (x_i\le t-1)\text{~~and~~}  (z_{it}\ge 1) ~\Longrightarrow~  (x_i\ge t)$$
and therefore any \bb tree $\TT_\B$ can be constructed by branching only on the auxiliary variables.
Consequently,  any leaf node $N$  of the \bb tree has the form
$$N=\{z_{it}=0,~\forall (i,t)\in S_0,~\text{and},~ z_{it}=1,~\forall (i,t)\in S_1\}$$
for some  index sets $S_0$ and $S_1$.
Now consider a \bb tree $\TT$ for $P^I$ constructed from $\TT_\B$ as follows:
if two node in $\TT_\B$ are created from their common predecessor by adding the conditions $(z_{it}=0)$ and $(z_{it}=1)$, then  we create two nodes in $\TT$ by adding the conditions $ (x_i\le t-1)$ and $(x_i\ge t)$, respectively.
Note that for every leaf node $N$ of $\TT_\B$, there is a corresponding leaf node  $N'$ of $\TT$:
$$N'=\{x_i\le t-1,~\forall (i,t)\in S_0,~\text{and},~ x_i\ge t,~\forall (i,t)\in S_1\}=\{a_i\le x_i\le b_i,~\forall i\in I\}$$
for some integer vectors $a$ and $b$.

Given a point $\bar x\in P\cap N'$, we construct a point $(\bar x,\bar z)\in P_\B$ where for all $i\in I$
$$\bar z_{it}=\left\{ 
\begin{array}{cl} 1 & 1\le t\le a_i\\[.1cm] ({\bar x_i-a_i})/({b_i-a_i})& a_i< t\le b_i\\[.1cm]0 &  b_i<t\le u_i.\end{array}   \right.$$
It is easy to see that the  point $(\bar x,\bar z)\in N$ and therefore $\proj_x(N \cap P_\mathcal{B})\supset (N \cap P)$.
Consequently, branching on the auxiliary binary variables associated with the unary binarization scheme does not seem useful.

Now consider an alternative binarization defined by the binarization polytope
\begin{eqnarray}
  \BALT(u) = \{(x,z)\in\R\times[0,1]^{u}\::&&\:\textstyle x=u z_u + \sum_{j=1}^{u-1}  z_j,\\ \nonumber
  &&0\le z_{u-1}\le \ldots \le z_{1} \le 1, z_1 + z_u \leq 1\}, \label{bin-alt}
\end{eqnarray}
and  the polyhedron 
$$
P^n = \Big\{ x \in [0,4]^n : \sum_{i \in S} x_i  + \sum_{i \notin S} (4- x_i) \geq \frac{1}{2},~~ \forall S \subseteq \{1,\ldots,n\}\Big\}
$$
obtained by cutting all of the corners of the hypercube $[0,4]^n$. Clearly
\begin{equation}
P^n\cap \Z^n = \{0,1, 2,3,4\}^n \setminus \{0,4\}^n.
\end{equation}

\begin{proposition}\label{prop:bbt1}
For the binarization $\B= (\BALT(4), \dots, \BALT(4))$, there exists a complete \bbt~$\TT_\B$ with respect to the binary extended formulation $P^n_\B$ with size $2^n + n$.
\end{proposition}
\begin{proof}
Consider the \bb tree $\TT_\B$ constructed as follows: 
We label the root node with 
$$L_0 = \{(x,z)\in[0,4]^n \times [0,1]^{n\times 4}\::\: (x_i,z_i)\in \BALT(4)\text{ for }i=1,\ldots,n\}$$ 
where $z_i$ denotes  the vector of auxiliary variables associated with $x_i$.
For $i=1, \dots, n$, node $L_{i-1}$ has two successor nodes  $L_i$ and $ R_i$ obtained by branching on variable $z_{i1}$ as follows: 
\beqa
L_i &=& \{ (x,z) \in L_{i-1} : z_{i1} = 0 \}, \ \ \ R_i = \{(x,z) \in L_{i-1} : z_{i1} = 1\}
\eeqa
Nodes $R_1,\ldots, R_n$ are leaf nodes of the tree.
Note that as $z_{i1}=1$ for $(x,z)\in R_i$,
\beqa R_i&=& \{(x,z)\in L_0\::\: z_{k1}=0 \text{ for } k<i, ~z_{i1}=1, ~x_i=1+z_{i2}+z_{i3}\}.
\eeqa
Consequently
$$\proj_x(P^n_\B\cap R_i ) =  \{x\in P^n\::\: x_k\in [0,4] \text{ for } k\not=i,~x_i\in[1,3]\}\subseteq \conv(P^n\cap \Z^n).$$
The rest of the tree consists of a complete binary tree of depth $n$ rooted at node $L_n$  obtained  by branching on $z_{i4}$ for all $i=1,\ldots,n$.
This leads to $2^n$ additional leaf nodes 
\beqa N_S &=& \{ (x,z)\in  L_n  :  z_{i4} = 1 \ \forall i \in S, z_{i4} = 0 \ \forall i \notin S\},\\
&\subseteq& \{ (x,z)\in  L_n  :  x_{i} = 4 \ \forall i \in S, x_{i} = 0 \ \forall i \notin S\}
\eeqa
one for each subset $S$ of $\{1,\ldots,n\}$, and notice that $P^n_\B\cap N_S=\emptyset$. 
The tree $\TT_\B$ has a total of $2^n + n$ leaf nodes, see Figure~\ref{fig:uno}.
Therefore, 
$$\conv\Big(\bigcup_{N \in\leaf( \TT_\B)} \proj_x(P^n_\B \cap N)\Big )  
= \conv\Big(\bigcup_{i=1}^n  \proj_x(P^n_\B\cap R_i )\Big) 
\subseteq \conv(P^n \cap \Z^n),$$
implying $\TT_\B$ is complete, see Figure~\ref{fig:dos}.  \qed
\end{proof}

\begin{figure}[tb]
	\begin{center}
		\usetikzlibrary{arrows}

\tikzset{
  treenode/.style = {align=center, inner sep=0pt, text centered,    font=\sffamily},
  bucket/.style = {treenode, rectangle, white, font=\sffamily\bfseries, draw=black,    fill=black, text width=1.5em, text height=1.1em},
  bucketB/.style = { rectangle, rounded corners,	draw, align=center, draw=black},
  decision/.style = {treenode, circle, black, draw=black,     text width=1.5em, very thick},
  decisionT/.style = {treenode, circle, white, draw=black,   fill=black,  text width=1.5em, very thick},
  topstart/.style = {treenode,   minimum width=0.5em, minimum height=1.5em}
}

\begin{tikzpicture}[scale=0.5,->,>=stealth'] 
 \tikzstyle{level 1}=[sibling distance=12cm]
 \tikzstyle{level 2}=[sibling distance=8cm]
 \tikzstyle{level 3}=[sibling distance=7cm]
 \tikzstyle{level 4}=[sibling distance=3cm]
\node [decision] {$L_0$}
    child{ node [decision] {$L_1$} 
    	child{ node [decision] {$L_2$} 
         	child{ node [bucketB] { }
			child{ node [bucketB] {$N_{\emptyset}$}
				edge from parent node[above left]  {$z_{24} = 0$} 
				}		
			child{ node [bucketB] {$N_{\{2\}}$}
				edge from parent node[above right]  {{ $z_{24} = 1$}} 
				}
			edge from parent node[above left]  {$z_{14} = 0$} 
			}
		child{ node [bucketB] {} 
			child{ node [bucketB] {$N_{\{1\}}$}
			edge from parent node[above left]  {\hspace{-0.2cm}{$z_{24} = 0$}} 
			}		
			child{ node [bucketB] {$N_{\{1,2\}}$}
			edge from parent node[above right]  {$z_{24} = 1$} 
			}
			edge from parent node[above right]  {$z_{14} = 1$} 
			} 
		 edge from parent node[above left]  {$z_{21} = 0$} 
        		}
	child{ node [bucketB] {$R_2$}
		edge from parent node[above right]  {$z_{21} = 1$} 
            	}         
    edge from parent node[above left]  {$z_{11} = 0$} 
    }
    child{ node [bucketB] {$R_1$}
	edge from parent node[above right]  {$z_{11} = 1$} 
          }
 ; 
\end{tikzpicture}  
	\end{center}
	\caption{The tree $\TT_\B$ for $n=2$.}
	\label{fig:uno}
\end{figure}

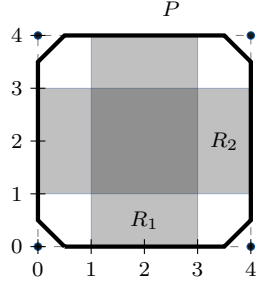
\begin{figure}[ht]
	\begin{center}
		\usetikzlibrary{arrows}

\begin{tikzpicture}[scale=0.7] 
\definecolor{myColor}{rgb} {0.00,0.33,0.68}
\colorlet{myColor2}{myColor!10}

\tikzset{
    myStyle/.style = {draw=dkblue, very thick, rectangle}
}

\draw[dashed, gray] (0,0) -- (4,0) -- (4,4) -- (0,4) -- (0,0);

\draw[draw = myColor, fill = black, opacity=0.25] (0,1) -- (4,1) -- (4,3) -- (0,3) -- (0,1);
\draw[draw = myColor, fill = black, opacity=0.25] (1,0) -- (1,4) -- (3,4) -- (3,0) -- (1,0);

fill = black, opacity=0.1

\node at (2,0.5)   (a) {$R_1$};
\node at (3.5,2)   (b) {$R_2$};

\foreach \x in {0,4}
	\foreach \y in {0,4}
		\draw[draw = myColor, fill = black, opacity=0.9] (\x,\y) circle (2pt);
		
\node at (2.5,4.5)   (a) {$P$};
\draw[line width=1.6pt] (0.5,0)--(3.5,0)--(4,0.5)--(4,3.5)--(3.5,4)--(0.5,4)--(0,3.5)--(0,0.5)--(0.5,0);
\foreach \x in {0,1,2,3,4}
    \draw (\x cm,4pt) -- (\x cm,-4pt) node[anchor=north] {$\x$};
\foreach \y in {0,1,2,3,4}
    \draw (4pt,\y cm) -- (-4pt,\y cm) node[anchor=east] {$\y$};
  
\end{tikzpicture}
	\end{center}
	\caption{The projection of  $R_i\cap P^n_\B$ for $i=1,2$, to $\R^2$.}
	\label{fig:dos}
\end{figure}

We next show that any complete \bbt\ in the original space is approximately at least twice as big as the one described in Proposition \ref{prop:bbt1}.
We will use the following  fact from convex analysis in the proof of the next claim:
If $a^T x \leq b$ is a valid inequality for $X \subseteq \R^n$, then 
\begin{equation}
\conv(X \cap \{ x : a^T x = b\}) = \conv(X) \cap \{a^Tx = b\}.
\label{eq:conv-H}
\end{equation}
%
%

\begin{proposition} Any complete \bbt~$\TT$ of $P^n$ has size at least $2\cdot2^n -1$.
\end{proposition}
\begin{proof}
Let $g(n)$ be the minimum size of a complete branching tree for $P^n$.  
We will prove that $g(n) \geq 2 \cdot 2^{n} -1$ by induction on $n$.  
For $n =1$, $P^1$ is the line segment $[0.5,3.5]$. Clearly, a single branch (with 2 leaves) does not lead to a complete tree and therefore  $g(1) \ge 3 = 2\cdot2^1 -1$.  We now assume $n \geq 2$, and assume the result holds for $P^k$ with $k < n$.

Let $\TT$ be a complete branching tree for $P^n$ with root node label $D = [0,4]^n$.
As $P^n$ is  symmetric we can rename the variables and assume that the first variable branched on is $x_n$ and the successor nodes are $L = \{ x\in D : x_n \leq t\}$ and $R = \{ x\in D : x_n \geq t+1\}$ for some $t \in \{0,1,2,3\}$.  
Any other choice of $t$ would lead to one of $L$ or $R$ being equal to $D$.
In addition, as the hyperplane defined by $x_n=2$ is a plane of symmetry for $P^n$, we can also assume that $t \in\{0,1\} $. Let $\TT_{L}$ be the subtree of $\TT$ rooted at $L$ and let $\TT_{R}$ be the subtree of $\TT$ rooted at $R$. 

Consider first $\TT_L$. 
Let $H = \{ x\in\R^n : x_n = 0\}$ and let $\TTp$  be the tree with the same choice of branches as $\TT_{L}$ but with root node $L \cap H$ 
Thus, for each node $N \in \TT_{L}$, there is a corresponding node in $\TTp$ with the label $N \cap H$.
If the successor nodes of node $A\in\TTp$ are obtained by branching on $x_n$ to $x_n = 0$ and $x_n = 1$, then clearly the left successor node has the same label as $A$ and the right one has the label $\emptyset$. 
Therefore, it is possible to replace the branching conditions on $x_n$ in $\TTp$ with $x_1\le4$ and $x_1\ge5$ to obtain the same labels.
Consequently, one can  obtain a new tree $\tilde\TT$, with identical labels at every node as $\TTp$, that branches only on variables $x_1,\ldots,x_{n-1}$.
We will next show that $\tilde\TT$ (and therefore $\TT_{L}$) has at least $g(n-1)$ leaf nodes.


As $\TT$ is a complete branching tree for $P^n$, it follows that
$ \conv(P^n \cap \Z^n) = \conv(\bigcup_{N \in \leaf(\TT)} N \cap P^n).$
  Intersecting both the left-hand and right-hand terms of the above equation with $H$, and then using equation~\eqref{eq:conv-H}
  to take $H$ inside the convex hull expressions, we obtain
\begin{equation}\label{eqxx}
  \conv(P^n \cap \Z^n \cap H) = \conv(\bigcup_{N \in \leaf(\TT_L)} N \cap (P^n \cap H)).
\end{equation}
The equality above follows from the fact that the intersection of the label of any leaf node of $\TT_R$ with $H$ is the empty set.
For each leaf node $N$ of $\TT_L$, the corresponding leaf node of $\TTp$ is $N \cap H$, and therefore (\ref{eqxx}) implies
  \begin{align*}
\conv((P^n \cap H) \cap \Z^n) = \conv(\bigcup_{N' \in \leaf(\TTp)} N' \cap (P^n \cap H)).
  \end{align*}
 Therefore $\TTp$ is a complete branching tree for $P^n \cap H=P^{n-1}\times \{0\}$ and so is $\tilde\TT$ as both $\tilde\TT$ and $\TTp$ have the same leaf node labels. 
Note that  $P^n\cap\Z^n\cap H =(P^{n-1}\cap \Z^{n-1})\times \{0\} $ and  $\tilde\TT$ only branches on variables $x_1, \dots, x_{n-1}$ and consequently,  $\tilde\TT$ yields a complete branching tree for $P^{n-1}$ after dropping $x_n$.  
Therefore $\tilde\TT$ has at least $g(n-1)$ leaf nodes implying that $\TT_L$ also has at least  $g(n-1)$ leaf nodes.

We now consider $\TT_{R}$, the second part of the tree $\TT$, which is rooted at $R$. 
We will next show that $\TT_{R}$ has at least $g(n-1)+1$ leaf nodes.
In this part of the proof, we let $H = \{ x\in\R^n : x_n = 4\}$ and let $\TTp$ be obtained from $\TT_{R}$ by changing the label of its root node to $R \cap H$. 
Repeating the same arguments used for $\TT_{L}$ earlier, it is easy to see that $\TTp$ has at least  $g(n-1)$ leaf nodes.
Moreover, note that $p^1 ,p^2 \in P^n\cap R$ where $p^1 = (0, 0, \dots, 3)$,  $p^2 = (0, 0, \ldots,3.5)$. 
As $\TT$ is complete and $p^1$ is integral, $\TT_{R}$ has a leaf node containing $p^1$. 
Furthermore, this leaf node cannot contain $p^2$ as it does not belong to $\conv(P^n\cap\Z^n)$.
Notice that the points $p^1$ and $p^2$ only differ in the last coordinate and therefore cannot be separated by a branching decision that involves the first $n-1$ variables.
Consequently, one of the branching conditions in $\TT_R$ (leading to this leaf node) must be on the variable $x_n$.
Therefore, the tree $\TTp$ must contain a node $\bar N$ whose successors are 
labeled $\emptyset$ and $\bar N$. 
Clearly, contracting the edge between these two nodes with the label $\bar N$ and deleting the node labeled $\emptyset$ still yields a complete tree for $P^n\cap H$ with at least one less leaf node than $\TTp$. Therefore, $\TTp$ has at least $g(n-1)+1$ leaf nodes as desired.


Combining the bounds on the leaf nodes of $\TT_L$ and $\TT_{R}$, we conclude that $g(n)  \geq 2g(n -1) + 1 \geq 2 \cdot (2^{n} -1) + 1 = 2 \cdot 2^n -1$.\qed
\end{proof}

\section*{Acknowledgements}
We would like to thank  Andrea Lodi for fruitful discussions on binarization.

\bibliographystyle{amsplain}
\providecommand{\bysame}{\leavevmode\hbox to3em{\hrulefill}\thinspace}
\providecommand{\MR}{\relax\ifhmode\unskip\space\fi MR }
\providecommand{\MRhref}[2]{%
  \href{http://www.ams.org/mathscinet-getitem?mr=#1}{#2}
}
\providecommand{\href}[2]{#2}

\newpage
\section*{Appendix}
\subsection*{\bf  Proof of the second part of Theorem \ref{thm:LG}.}

We now prove that the point $\bar x = (6/5,6/5)\in P$ belongs to $\projk{x,y}{SC(P_{LG})}$.
We first show that  $\bar x\in SC(P)$. Let
\begin{equation*}
p_1 = \left(\begin{array}{c}1 \\ 3/2\end{array}\right),~~~~
p_2 = \left(\begin{array}{c}3/2 \\ 3/2\end{array}\right),~~~~
p_3 = \left(\begin{array}{c}3/2 \\ 1\end{array}\right),
\end{equation*}
and note we can write $\bar x$ as a convex combination of any one of these points and an integral point in $P$, see Figure \ref{fig:proofbypicture2}.
More precisely: $\bar x = 4/5p_1+1/5(2,0) = 2/5p_2+3/5(1,1)= 4/5p_3+1/5(0,2)$.
Therefore, if  $\bar x \not\in \SC(P)$, then  for some split set $S$ we have $\bar x\not\in \conv(P \setminus S)$ and  $p_1,p_2,p_3,\bar x \in S$. Let $\R^2\setminus S=A\cup B\supset\Z^2$ where $A$ and $B$ are half spaces denoting the two sides of the split disjunction. Without loss of generality, assume $(1,1)\in A$.
As $p_1\not \in A$, we have $(2,1)\in B$ and as $p_3\not \in A$, we have $(1,2)\in B$.
But then, $p_2\in B$ as $p_2=1/2(2,1)+1/2(1,2)$ and $\bar x\in B$, a contradiction.
Therefore, $\bar x\in SC(P)$.

\begin{figure}[h]\begin{center}\begin{tikzpicture}[scale=1.5] 
		\draw[dashed, , opacity=0.5] (0,0) -- (2,0) -- (2,2) -- (0,2) -- (0,0);
		\draw [ ->](0,0) -- (0,2.2); \draw [ ->](0,0) -- (2.5,0); 
		
		\draw [fill = black, opacity=0.1] (0,0) -- (2,0) -- (1.81,1.81) -- (0,2) -- (0,0);
		\draw  (0,0) -- (2,0) -- (1.81,1.81) -- (0,2) -- (0,0);
		
		\foreach \x in {0,1,2}\foreach \y in {0,1,2}
		\draw[fill = white] (\x,\y) circle (.75pt);
		
		\foreach \x in {0,1,2} \draw (\x cm,.1pt) -- (\x cm,-.1pt) node[anchor=north] {$\x$};
		\foreach \y in {0,1,2} \draw (.1pt,\y cm) -- (-.1pt,\y cm) node[anchor=east] {$\y$};
		
		\draw[fill = black] (1,3/2) circle (.65pt); \node at (1,1.65) {$p_1$};
		\draw[fill = black]  (3/2,3/2) circle (.65pt); \node at (1.65,1.65) {$p_2$};
		\draw[fill = black] (3/2,1) circle (.65pt);\node at (1.7,1.1) {$p_3$};
		\draw[fill = black] (6/5,6/5) circle (.65pt); \node at (1.2,1) {$\bar x$};
		
		\draw[dashed] (1,3/2) -- (2,0);\draw[dashed] (0,2) -- (3/2,1); \draw[dashed]  (3/2,3/2)-- (1,1); 
		
		\end{tikzpicture}\end{center}
	\caption{Polytope $P$}\label{fig:proofbypicture2}
\end{figure}
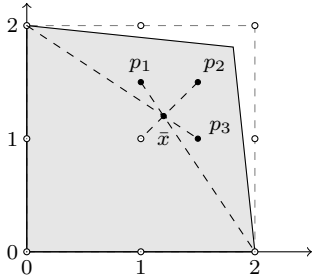

To prove that  $\bar x$ belongs to $\projk{x,y}{SC(P_{LG})}$, we will show that $\bar p = (\bar x, \bar z)\in \SC(P_{LG})$ where
\begin{equation*}
\bar x = \left(\begin{array}{c}6/5 \\ 6/5\end{array}\right), \bar z = \left(\begin{array}{cc}2/5 & 2/5 \\ 2/5 & 2/5\end{array}\right).
\end{equation*}
As $\bar x\in P$ and $\bar p$ satisfies $x_i = z_{i1} + 2 z_{i2}$ for $i=1,2$, we have $\bar p \in P_{LG}$.
Suppose $\bar p \not\in \SC(P_{LG})$.
Then  $\bar p\not\in \conv(P_{LG} \setminus S)$ for some split set $S = \{(x,z) \in\R^{2+4}: d < a^Tx + c\cdot z < d+1\}$ where $c$ is an integral matrix, $a$ is an integral vector and $d$ is an integer, and $c \cdot z = \sum_{ij} c_{ij}z_{ij}$.
Subtracting appropriate multiples of the equations $x_i = z_{i1} + 2z_{i2}$ (valid for $P_{LG}$) from $a^Tx + c \cdot z$, we can assume $a = 0$ and 
$$S=\big\{(x,z) : d < c\cdot z< d+1\big\}$$
for some nonzero $c$.
As $c$ is integral and $\bar z$ is (1/5)-integral, it follows that $c \cdot \bar z = d + \delta$
where $\delta \in\{1/5, 2/5, 3/5,4/5\} $.

We will next construct several pairs of points $p',p''\in P_{LG}$  with the property that $\bar p$ is a convex combination of $p', p''$; if both $p'$ and $p''$ are not contained in $S$, then $\bar p \in \conv(P_{LG} \setminus S)$, a contradiction.
Therefore the split set must contain at least one of $p'$ or $p''$ for each pair, and we will use this fact to impose conditions on  $c$ till we get a contradiction.

First note that if $3/5 \leq x_1 \leq 9/5$ and $3/5 \leq x_2 \leq 9/5$, then $(x_1,x_2) \in P$.
Therefore, for any such $(x_1, x_2)$, choosing $z_{ij}$ values for $i=1,2$ and $j=1,2$ such that $x_1 = z_{11} + 2z_{12}$ and $x_1 = z_{21} + 2z_{22}$,
we get a point in $P_{LG}$.

(i)
Let $d^1 = (d^1_x, d^1_z)$ and $ d^2 = (d^2_x, d^2_z)$  where
\begin{equation}\label{logd1} d^1_x = \left(\begin{array}{c}2/5 \\ 0\end{array}\right), ~d^1_z = \left(\begin{array}{cc}2/5 & 0\\ 0 & 0\end{array}\right) \mbox{ and } d^2_x = \left(\begin{array}{c}0 \\ 2/5\end{array}\right),~ d^2_z = \left(\begin{array}{cc}0 & 0\\ 2/5 & 0\end{array}\right).
\end{equation}
Let $d = d^1 + d^2$ and consider the pair of points $p'=p + d$ and $p''=p - d$ in $P_{LG}$.
Clearly, $\bar p=p'/2+p''/2$,
and for $p'$,  we have $$ c \cdot z'  = c \cdot \bar z + c\cdot (d^1_z+d^2_z) = d+\delta + 2/5(c_{11}+c_{21}),$$
and for $p''$, we have $$ c \cdot z'' = c \cdot \bar z - c\cdot (d^1_z+d^2_z) = d+\delta - 2/5(c_{11}+c_{21}).$$
Therefore, unless $|c_{11} + c_{21}| \leq 1$ both $p',p''$ lie outside the split set $S$ (recall that $1/5 \leq \delta \leq 4/5$).
Therefore, we conclude that $|c_{11} + c_{21}| \leq 1$.
Similarly, letting $d= d^1  - d^2$, we conclude that $|c_{11} - c_{21} | \leq 1$, which implies that $|c_{11}|+|c_{21}| \leq 1$.
Furthermore, as both $P$ and $p$ are symmetric with respect to the coordinates $x_1$ and $x_2$, we can assume that $|c_{11}| \geq |c_{21}|$ and therefore, $|c_{21}|=0$ and $|c_{11}| \leq 1$.
As $\bar x\in\SC(P)$, Proposition \ref{prop:individual-split} implies that $c_{22}\not=0$.


(ii)
Now consider $d^3 = (d^3_x, d^3_z)$ and $ d^4 = (d^4_x, d^4_z)$  where 
\begin{equation}\label{logd3} 
d^3_x = \left(\begin{array}{c}2/5 \\ 0\end{array}\right),~
d^3_z = \left(\begin{array}{cc}-2/5 & 2/5\\ 0 & 0\end{array}\right)\mbox{ and }
d^4_x = \left(\begin{array}{c}0 \\ 2/5\end{array}\right), ~
d^4_z = \left(\begin{array}{cc}0 & 0 \\ -2/5 & 2/5\end{array}\right).
\end{equation}
 Letting $d = d^1 + d^4$ and considering the pair of points $p'=p + d$ and $p''=p - d$ in $P_{LG}$,
we can now argue that 
\beqn |c_{11} + c_{22} - c_{21}| \leq 1.\label{logc1} \eeqn
Similarly, using $d = d^1 - d^4$ we conclude that 
\beqn|c_{11} - c_{22} + c_{21}| \leq 1 .\label{logc2} \eeqn
As  $|c_{21}|=0$, inequalities \eqref{logc1} and \eqref{logc2} 
together imply that $|c_{11}| + |c_{22} | \leq 1 $.
As  $|c_{22} | \geq 1 $ we conclude that  $|c_{11}|=0 $ and $ |c_{22} | = 1 $.
Furthermore, as $|c_{11}|=0 $ we observe that $ |c_{12} | \not = 0 $ by Proposition \ref{prop:individual-split}.

(iii)  Letting $d = d^3 + d^4$ and using points $p'=p + d$ and $p''=p - d$ in $P_{LG}$, we can argue that 
\beqn |c_{12} - c_{11} + ( c_{22} - c_{21})| \leq 1 (\mbox{from }d = d^3 + d^4), \label{logc3} \eeqn
and letting $d = d^3 - d^4$, similarly, we can argue that 
\beqn |c_{12} - c_{11} - ( c_{22} - c_{21})| \leq 1 (\mbox{from }d = d^3 - d^4). \label{logc4}\eeqn
As $ c_{11} = c_{21}=0$, these inequalities simplify to $|c_{12} +  c_{22} | \leq 1 $ and $|c_{12} -  c_{22} | \leq 1 $.
Consequently $|c_{12}| + | c_{22} | \leq 1 $ which gives the desired contradiction as $| c_{22} | =1$ and $|c_{12}|\neq 0 $.\qed

\subsection*{\bf Proof of the second part of Theorem \ref{thm:U}.}
We now prove that the point $(\bar x,\bar y) = [(1.5,1,1.5),(.5,.5,.5)]$ belongs to $\projk{x,y}{SC(P_{LG+})}$.
We will show that $\bar p = (\bar x, \bar y, \bar z)\in \SC(P_{LG+})$ where
\begin{equation*}
\bar x = \left(\begin{array}{c}1.5 \\ 1 \\ 1.5\end{array}\right), \bar y = \left(\begin{array}{cc}.5 \\ .5 \\ .5\end{array}\right), \bar z = \left(\begin{array}{cc}.5 & .5 \\ 0 & .5 \\ .5 & .5\end{array}\right).
\end{equation*}

It is easy to verify that $\bar p \in P_{LG+}$. Suppose $\bar p \not\in \SC(P_{LG+})$.
Then  $\bar p\not\in \conv(P_{LG+} \setminus S)$ for some split set $S = \{(x,y, z) \in\R^{3+3+6}: d < a^Tx + b^Ty  + c\cdot z < d+1\}$ where $c$ is an integral matrix, $a,b$ are integral vectors and $d$ is an integer, and $c \cdot z = \sum_{ij} c_{ij}z_{ij}$.
As in the proof of Theorem \ref{thm:LG}, we can argue that 
$a = 0$ and 
$$\bar p \in S=\big\{(x,y,z) : d < b^Ty + c\cdot z< d+1\big\}.$$
As $b$ and $c$ are integral and $\bar y$ and $\bar z$ are half-integral, it follows that $b^T\bar y  + c \cdot \bar z$ is half-integral and  $$b^T\bar y  + c \cdot \bar z = d + 0.5.$$
Moreover, all points in $P_{LG+}$ satisfy $x_1 + x_2 + x_3 = 4$  and therefore $\sum_{i=1}^3 z_{i1} + 2\sum_{i=1}^3 z_{i2} = 4$.
We can add multiples of this equation to $b^Ty + c\cdot z$ to eliminate the coefficient of $z_{11}$. 
Therefore without loss of generality, we can  assume that $c_{11}=0$. 

We next construct several pairs of points $p',p''\in P_{LG+}$ such that $\bar p=0.5 p'+0.5 p''$; then if both $p'$ and $p''$ lie outside $S$, then $\bar p \in \conv(P_{LG+} \setminus S)$, a contradiction.
Therefore $S$ must contain at least one of $p'$ or $p''$ for each pair, and we will use this fact to impose conditions on $b$ and $c$, till we show that there cannot exist such a split set.
 
(i)
Consider the pair of points $p'=(x', \bar y, z')$ and $p''=(x'', \bar y, z'')$ in $P_{LG+}$ defined by
\[ x' = \left(\begin{array}{c}2 \\ 1 \\ 1\end{array}\right), z' = \left(\begin{array}{cc}1 & .5 \\ 0 & .5 \\ 0 & .5\end{array}\right) \mbox{ and }  x'' = \left(\begin{array}{c}1 \\ 1 \\ 2\end{array}\right), z'' = \left(\begin{array}{cc}0 & .5 \\ 0 & .5 \\ 1 & .5\end{array}\right) \]
and note that $\bar p=0.5 p'+0.5 p''$.
For $p'$, we have $ b^T\bar y + c \cdot z' = b^T\bar y + c \cdot \bar z - .5(c_{31})=d+1/2-.5(c_{31})$ 
and for $p''$, we have $ b^T\bar y + c \cdot z'' = b^T\bar y + c \cdot \bar z + .5(c_{31})=d+1/2+.5(c_{31})$ and clearly unless $c_{31}=0$ both $p',p''$ lie outside the split set $S$.
Therefore, we conclude that 	  $c_{31} = 0$.

(ii) Next consider $p'=(x', \bar y, z')$ and $p''=(x'', \bar y, z'')$ in $P_{LG+}$ defined by
\[ x' = \left(\begin{array}{c}2 \\ 1 \\ 1\end{array}\right),  z' = \left(\begin{array}{cc}0 & 1 \\ 0 & .5 \\ 1 & 0\end{array}\right) \mbox{ and
} , x'' = \left(\begin{array}{c}1 \\ 1 \\ 2\end{array}\right),  z'' = \left(\begin{array}{cc}1 & 0 \\ 0 & .5 \\ 0 & 1\end{array}\right) \]
and note that $\bar p=0.5 p'+0.5 p''$.
For $p'$, we have $b^T\bar y + c \cdot z' = b^T\bar y + c \cdot \bar z + .5(c_{12} - c_{32})=d+1/2+ .5(c_{12} - c_{32})$ and for $p''$ we have $b^T\bar  y + c \cdot z'' = b^T\bar y + c \cdot \bar z - .5(c_{12} - c_{32})=d+1/2 - .5(c_{12} - c_{32})$.
Therefore, unless $c_{12} - c_{32}=0$, both $p',p''$ lie outside the split set $S$ and we conclude that $c_{12} = c_{32}$.

(iii) Next consider $p'=(\bar x, \bar y, z')$ and $p''=(\bar x, \bar y, z'')$ in $P_{LG+}$ defined by
\[	 z' = \left(\begin{array}{cc}0 & .75 \\ 0 & .5 \\ 0 & .75\end{array}\right) \mbox{ and }
  z'' = \left(\begin{array}{cc}1 & .25 \\ 0 & .5 \\ 1 & .25\end{array}\right). \]
For $p'$, we have $b^T\bar y + c \cdot z' = b^T\bar y + c \cdot \bar z + .25(c_{12} + c_{32})= d+1/2 + .5c_{12}$ 
and for $p''$ we have  $b^Ty'' + c \cdot z'' = d+1/2 - .5c_{12}$.
Unless $c_{12} =0$, both $p',p''\not\in S$ and we conclude that $c_{12} = c_{32}=0$.

(iv) Next consider $p'=(x', \bar y, z')$ and $p''=(x'', \bar y, z'')$ in $P_{LG+}$ defined by
	\[ x' = \left(\begin{array}{c}1 \\ 2 \\ 1\end{array}\right), z' = \left(\begin{array}{cc}1 & 0 \\ 0 & 1 \\ 1 & 0\end{array}\right) \mbox{ and }  x'' = \left(\begin{array}{c}2 \\ 0 \\ 2\end{array}\right), z'' = \left(\begin{array}{cc}0 & 1 \\ 0 & 0 \\ 0 & 1\end{array}\right) .\]
For $p'$, we have $b^T\bar y + c \cdot z' = d+1/2 + .5c_{22}$ 
and for $p''$, we have $b^Ty'' + c \cdot z'' = d+1/2 - .5c_{22}$.
Both $p'$ and $p''$ lie outside $S$ unless $c_{22} = 0$. Therefore,  $c_{22} = 0$.
	
(v) Next consider $p'=(x', y', z')$ and $p''=(x'',  y', z'')$ in $P_{LG+}$ defined by
	\[ x' = \left(\begin{array}{c}2 \\0\\2\end{array}\right),
	 y' =  \left(\begin{array}{c}.5\\ 0  \\  .5\end{array}\right),
	 z' = \left(\begin{array}{cc}0 & 1 \\ 0 & 0 \\ 0 & 1\end{array}\right) \mbox{ and } 
	 x'' = \left(\begin{array}{c}1 \\ 2\\  1 \end{array}\right),
	 y'' = \left(\begin{array}{c}.5 \\ 1\\  .5 \end{array}\right),
	  z'' = \left(\begin{array}{cc}1 & 0 \\ 0 & 1 \\ 1 & 0\end{array}\right). \]
For $p'$, we have $b^Ty' + c \cdot z' =d+1/2 + .5b_{2}$
and for $p''$, we have $b^Ty'' + c \cdot z'' = d+1/2 - .5b_{2}$.
Both $p'$ and $p''$ lie outside $S$ unless $b_{2} = 0$. Therefore, $b_2 = 0$.
		
(vi)  Next consider $p'=(x', y', z')$ and $p''=(x'',  y', z'')$ in $P_{LG+}$ defined by
\[ 
x' = \left(\begin{array}{c}3 \\ 0  \\ 1 \end{array}\right), 
y' = \left(\begin{array}{c}1\\  0 \\  .5\end{array}\right), 
z' = \left(\begin{array}{cc}1 & 1 \\ 0 & 0 \\ 0 & .5\end{array}\right) \mbox{ and } 
x'' = \left(\begin{array}{c}0  \\ 2  \\ 2 \end{array}\right), 
y'' = \left(\begin{array}{c}0 \\  1 \\  .5\end{array}\right), 
z'' = \left(\begin{array}{cc}0 & 0 \\ 0 & 1 \\ 1 & .5\end{array}\right) \]
For $p'$, we have  $b^Ty' + c \cdot z' =d+1/2 + .5b_{1}$ and for $p''$, we have $b^Ty'' + c \cdot z'' =d+1/2 - .5b_{1}$.
Both  $p'$ and $p''$ lie outside $S$ unless $b_{1} = 0$. Therefore, $b_1 = 0$. Similarly we can argue that $b_3 = 0$.

Combining these observations, we conclude that all components of $b$ and $c$ have to be zero except $c_{21}$. 
Therefore,  $b^T \bar y + c \cdot \bar z = c_{21}\bar z_{21} = 0$, and  we obtain a contradiction to the fact that $b^T\bar y + c \cdot \bar z$ is half integral, and the proof is complete.\qed

\end{document}